\documentclass[11pt]{article}
\usepackage{amsmath}
\usepackage{pdfsync}
\usepackage{graphicx,psfrag,epsf}
\usepackage{enumerate}
\usepackage{natbib}
\usepackage{url} 
\usepackage{pdfsync}
\usepackage{epsfig}
\usepackage{amssymb}
\usepackage[latin1]{inputenc}
\usepackage{amsthm}
\usepackage[english]{babel}
\usepackage{color}

\newcommand{\blind}{0}

{\catcode `\@=11 \global\let\AddToReset=\@addtoreset}
\AddToReset{equation}{section}

\newtheorem{theorem}{Theorem}[section]

\newtheorem{@definition}{\sc Definition}[section]

\newtheorem{@remark}{\sc Remark}[section]

\newtheorem{exemple}{\sc Example}[section]

\newcommand{\beqn}{\begin{displaymath}}
\newcommand{\eeqn}{\end{displaymath}}
\newcommand{\beq}{\begin{equation}}  
\newcommand{\eeq}{\end{equation}}

\def\mathsf{\bf}
\def\N{\mathbb{N}}

\def\R{\mathbb{R}}

\def\P{\mathrm P}

\def\text{\mbox}

\def\1{{\bf 1}}

\def\simn{\renewcommand{\arraystretch}{0.5}
\begin{array}[t]{c}
\stackrel{}{\sim} \\
{\scriptstyle n\rightarrow\infty}
\end{array}\renewcommand{\arraystretch}{1}}

\def\limiteloin{\renewcommand{\arraystretch}{0.5}
\begin{array}[t]{c}
\stackrel{{\cal D}}{\longrightarrow} \\
{\scriptstyle n \rightarrow\infty}
\end{array}\renewcommand{\arraystretch}{1}}

\def\limiteproban{\renewcommand{\arraystretch}{0.5}
\begin{array}[t]{c}
\stackrel{\P}{\longrightarrow} \\
{\scriptstyle n \rightarrow\infty}
\end{array}\renewcommand{\arraystretch}{1}}

\def\limiten{\renewcommand{\arraystretch}{0.5}
\begin{array}[t]{c}
\stackrel{}{\longrightarrow} \\
{\scriptstyle n\rightarrow\infty}
\end{array}\renewcommand{\arraystretch}{1}}

\def\limiteasn{\renewcommand{\arraystretch}{0.5}
\begin{array}[t]{c}
\stackrel{a.s.}{\longrightarrow} \\
{\scriptstyle n\rightarrow\infty}
\end{array}\renewcommand{\arraystretch}{1}}

\def\limitet{\renewcommand{\arraystretch}{0.5}
\begin{array}[t]{c}
\stackrel{}{\longrightarrow} \\
{\scriptstyle t\rightarrow\infty}
\end{array}\renewcommand{\arraystretch}{1}}

\def\limitet0{\renewcommand{\arraystretch}{0.5}
\begin{array}[t]{c}
\stackrel{}{\longrightarrow} \\
{\scriptstyle t\rightarrow 0}
\end{array}\renewcommand{\arraystretch}{1}}

\newtheorem{rem}{Remark}

\newtheorem{prop}{Proposition}

\date{}
\begin{document}
\def\spacingset#1{\renewcommand{\baselinestretch}%
{#1}\small\normalsize} \spacingset{1}

\if0\blind
{\title{\bf A new non-parametric detector of univariate outliers for distributions with unbounded support}
\author{Jean-Marc Bardet \\ 
and \\
Solohaja-Faniaha Dimby\thanks{
    The authors gratefully acknowledge the enterprise Autobiz}\hspace{.2cm}\\S.A.M.M., Universit\'{e} de Paris 1 Panth\'eon-Sorbonne\\90, rue de Tolbiac, 75634, Paris, France}
\maketitle

} \fi
{
\bigskip
\begin{abstract}
The purpose of this paper is to construct a new non-parametric detector of univariate outliers  and to study its asymptotic properties. This detector is based on a Hill's type statistic. It satisfies a unique asymptotic behavior for a large set of probability distributions with positive unbounded support (for instance: for the absolute value of Gaussian, Gamma, Weibull, Student or regular variations distributions). We have illustrated our results by numerical simulations which show the accuracy of this detector with respect to other usual univariate outlier detectors (Tukey, MAD or Local Outlier Factor detectors). The detection of outliers in a database providing the prices of used cars is also proposed as an application to real-life database.
\end{abstract}

\noindent%


 \thispagestyle{empty}
\newpage
\spacingset{1.45} 

\section{Introduction}
Let $(X_1,\cdots,X_n)$ be a sample of positive, independent, identically distributed random variables with unbounded distribution. This article aims to provide a non-parametric outlier detector among the "large" values of $(X_1,\cdots,X_n)$.
\begin{rem}
If we wish to detect outliers among the "small" values of $(X_1,\cdots,X_n)$, it would be possible to consider $\max(X_1,\cdots,X_n)-X_i$ instead of $X_i$, for $i=1,\cdots, n$. Moreover, if $X_i$, $i=1,\cdots,n$, are not positive random variables, as in the case of quantile regression residuals, we can consider $|X_i|$ instead of $X_i$. 
\end{rem}
There are numerous outlier detectors in such a framework. Generally, such detectors consist of statistics directly applied to each observation, deciding if this observation can be considered or not as an outlier (see for instance the books of Hawkins, 1980, Barnett and Lewis, 1994, Rousseeuw and Leroy, 2005, or the article of Beckman and Cook, 1983). The most frequently used, especially in the case of regression residuals, is the Student-type detector (see a more precise definition in Section \ref{simu}). However, it is a parametric detector that is theoretically defined for a Gaussian distribution. Another well-known detector is the robust Tukey detector (see for example Rousseeuw and Leroy, 2005). Its confidence factor is computed from quartiles of the Gaussian distribution, though it is frequently used for non-Gaussian distributions. Finally, we can also cite the $MAD_e$ detector using a confidence factor computed from the median of absolute value of Gaussian distribution (see also Rousseeuw and Leroy, 2005).  \\
Hence all the most commonly used outlier detectors are based on Gaussian distribution and they are not really accurate for heavier distributions (for regression residuals, we can also cite the Grubbs-Type detectors introduced in Grubbs, 1969, extended in Tietjen and Moore, 1972). Such a drawback could be avoided by considering a non-parametric outlier detector. However, in the literature there are few non-parametric outlier detectors. We could cite the Local Outlier Factor (LOF) introduced in Breunig {\it et al.} (2000), which is also valid for
multivariate outliers. For any integer $k$, the LOF algorithm compares the density of each point to the density of its $k$-closest neighbors. 
Unfortunately a theoretical or numerical procedure for choosing the number $k$ of cells and its associated threshold  still does not exist. There are also other detectors essentially based on a classification methodology (for instance: Knorr {\it et al.}, 2000) or robust statistics (for instance:  Hubert and Vendervieren, 2008, for univariate skewed data and Hubert and Van der Veeken, 2008, for multivariate ones). \\
An interesting starting point for defining a non-parametric detector of outlying observations is provided by the order statistics $(X_{(1)},\ldots, X_{(n)})$ of $(X_1,\cdots,X_n)$. Thus, Tse and Balasooriya (1991) introduced a detector based on increments of order statistics, but only for the exponential distribution.  Recently, a procedure based on the Hill estimator was also developed for detecting influential data points in Pareto-type distributions (see Hubert {\it et al.}, 2012). The Hill estimator (see Hill, 1975) has been defined from the following property: 
the family of r.v. $\big (j\, \big (\log(X_{(n-j+1)})-\log (X_{(n-j)})\big )\big )_{1\leq j\leq k(n)}$
 is asymptotically (when $\min \big (k(n)\ , \ n-k(n) \big ) \limiten  \infty$) a sample of independent r.v. following exponential distributions (R\'enyi exponential representation) for distributions in the max-domain of attraction of $G_\gamma$ where $G_\gamma$ is the cumulative distribution function of the extreme value distribution (see Beirlant {\it et al.}, 2004).  \\
Here we will use  an extension of this property for detecting a finite number of outliers among the sample $(X_1,\cdots,X_n)$. Indeed, an intuitive idea for detecting them is the following: the presence of outliers generates a jump in the family of r.v. $\big ( X_{(n-j+1)}/X_{(n-j)}\big )_j$, therefore we also see a jump in the family of the r.v. $\big (\log(X_{(n-j+1)})-\log (X_{(n-j)})\big )_j$. Thus an outlying data detector can be obtained when the maximum of this family exceeds a threshold (see details in (\ref{testT}) or (\ref{D})). In the sequel, see some assumptions on probability distributions for applying this new test of outlier presence. This also provides an estimator of the number of outliers. It is relevant to say that this test is not only valid for Pareto-type distribution (for instance Pareto, Student or Burr probability distributions), but more generally to a class of regular variations distributions. It can also be applied to numerous probability distributions with an exponential decreasing probability distribution function (such as Gaussian, Gamma or Weibull distributions). So our new outlier detector is a non-parametric estimator defined from an explicit threshold, which does not require any tuning parameter and can be applied to a very large family of probability distributions. \\
Numerous Monte-Carlo experiments carried out  in the case of several probability distributions attest to the accuracy of this new detector. It is compared to other famous outlier detectors or extended versions of these detectors and the simulation results obtained by this new detector are convincing especially since it ignores false outliers. Moreover, an application to real-life data (price, mileage and age of used cars) is done, allowing the detection of two different kinds of outliers. \\
~\\
We have drafted our paper along following lines. Section \ref{Main} contains the definitions and several probabilistic results while Section \ref{stat} describes how to use them to build a new outlier detector.  Section \ref{simu} is devoted to Monte-Carlo experiments, Section \ref{applis} presents the results of the numerical application on used car variables and the proofs of this paper are to be found in Section \ref{proofs}.
\section{Definition and first probabilistic results} \label{Main} 
For $(X_1,\cdots,X_n)$ a sample of positive i.i.d.r.v. with unbounded distribution, define:
\begin{eqnarray}\label{G}
G(x)=\P (X_1 >x)\qquad \mbox{for $x\in \R$.} 
 \end{eqnarray}
It is clear that $G$ is a decreasing function  and $G(x) \to 0$ when $x \to \infty$. Hence, define also the pseudo-inverse function of $G$ by  
\begin{eqnarray}\label{G1}
G^{-1}(y)=\sup \{x \in \R,~G(x)\geq y\}\qquad y\geq 0.
\end{eqnarray}
$G^{-1}$ is also a decreasing function. Moreover, if the support of the probability distribution of $X_1$ is unbounded then $G^{-1}(x) \to \infty$ when $x \to 0$. \\
~\\
Now, we consider both the following spaces of functions:
\begin{itemize}
\item $A_1=\Big \{ f:\,[0,1] \to \R$, such as for any $\alpha >0$, $\displaystyle f(\alpha x)=f_1(x)\Big (1+\frac {f_2(\alpha)}{\log(x)}+O\big (\frac 1 {\log^2(x)}\big ) \Big )$ when $x \to 0$ 
where $f_1:\,[0,1] \to \R$ satisfies $\lim_{x \to 0}f_1(x)=\infty$ and $f_2$ is a ${\cal C}^1([0,\infty))$ diffeomorphism $\Big \}$.
\item $A_2=\Big \{ g:\,[0,1] \to \R$, there exist $a>0$ and a  function $g_1:\,[0,1] \to \R$ satisfying $\lim_{x\to 0}g_1(x)=\infty$, and for all $\alpha >0$, $g(\alpha x)=\alpha^{-a} \, g_1(x)\, \big (1+O\big (\frac 1 {\log(x)}\big ) \big )~\mbox{when $x \to 0$}\Big \}$.
\end{itemize}
\begin{exemple}
We will show below that numerous famous "smooth" probability  distributions such as absolute values  of Gaussian, Gamma or Weibull distributions satisfy  $G^{-1} \in A_1$. Moreover, numerous heavy-tailed distributions such as Pareto, Student or Burr distributions have $G^{-1} \in A_2$. 
\end{exemple}
Using the order statistics $X_{(1)} \leq X_{(2)} \leq \cdots \leq X_{(n)}$, define the following ratios $(\tau_j)$ by:
\begin{eqnarray}
\label{rapport}& &\displaystyle \tau_j=\frac{X_{(j+1)}}{X_{(j)}} \quad \mbox{if $X_{(j)}>0$, and $\tau_j=1$ if not, for any $j=1,\cdots,n-1$} \\
\label{rapport2} & &\displaystyle  \tau'_j=(\tau_j-1) \log (n)\qquad \mbox{for any $j=1,\cdots,n-1$}
\end{eqnarray}
In the sequel, we are going to provide some probabilistic results on the maximum of these ratios.
\begin{prop}\label{prop1}
Assume $G^{-1} \in A_1$. Then, for any $J \in \N^*$, and with $(\Gamma_i)_{i \in \N^*}$ a sequence of r.v. satisfying $\Gamma_i=E_1+\cdots+E_i$ for $i \in \N^*$ where $(E_i)_{i\in \N^*}$ is a sequence of i.i.d.r.v. with exponential distribution of parameter $1$,
\begin{eqnarray}\label{asympto1}
\max_{j=1,\cdots,J} \{\tau_{n-j}'\} \limiteloin \max_{k=1, \cdots,J} \big \{ f_2 (\Gamma_{k+1})-f_2 (\Gamma_{k}) \big \}.
\end{eqnarray}
Here $f_2$ corresponds to $G^{-1}$ in the manner described in the definition
of $A_1$.
\end{prop}
\noindent Now, we consider a particular case of functions belonging to $A_1$. Let $A'_1$ be the following function space:
$$
A'_1=\big \{ f \in A_1\mbox{ and there are $C_1, C_2 \in \R$ satisfying $f_2(\alpha)=C_1+C_2 \log \alpha$ for all $\alpha >0$} \big \}.
$$
\begin{exemple}
Here there are some examples of classical probability distributions satisfying $G^{-1} \in A'_1$:
\begin{itemize}
\item {\bf Exponential distribution ${\cal E}(\lambda)$:} In this case,  $G^{-1}(x)=-\frac 1 \lambda \,\log (x)$, and this implies $G^{-1} \in A'_1$ with $f_1(x)=-\frac 1 \lambda \, \log (x)$ and $f_2(\alpha)=\log \alpha$ ($C_1=0$ and $C_2=1$). 
\item {\bf Gamma distributions $\Gamma(a)$} In this case, $G(x)=\frac 1 {\Gamma(a)} \int_x^\infty t^{a-1}e^{-t} dt$ for $a\geq 1$ and we obtain, using an asymptotic expansion of the incomplete gamma function (see Abramowitz and Stegun, 1964):
$$
G^{-1}(x)=\frac 1 {\Gamma(a)} \, \big (-\log x +(a-1)\log (-\log x)\big ) + O(|(\ln x )^{-1} |)\qquad x \to 0.
$$
As a consequence, we deduce $G^{-1} \in A'_1$ with
$$
f_1(x)= \frac 1 {\Gamma(a)} \, \big (-\log x +(a-1)\log (-\log x)\big )\quad \mbox{and} \quad f_2(\alpha)=\log \alpha~~ \mbox{($C_1=0$ and $C_2=1$)}.
$$ 
\item {\bf Absolute value of standardized Gaussian distribution $|{\cal N}(0,1)|$:} In this case, we can write  $G(x)=\frac 2 {\sqrt{2\pi}} \int_x^\infty e^{-t^2/2} dt=\mbox{erfc}(x/\sqrt 2)$, where $erfc$ is the complementary Gauss error function. But we know (see for instance Blair {\it et al.}, 1976) that for $x \to 0$, then $\mbox{erfc}^{-1}(x) =\frac 1 {\sqrt 2} \,  \Big ( -\log(\pi x^2)-\log (-\log x) \Big )^{1/2}+ O(|(\ln x )^{-1}| )$. As a consequence, for any $\alpha>0$,
\begin{equation}\label{erfc}
\mbox{erfc}^{-1}(\alpha \, x) =\mbox{erfc}^{-1}(x) \Big ( 1 + \frac {\log \alpha } {2 \log x } +  O(|(\ln x )^{-2}| )  \Big )\qquad x \to 0.
\end{equation}
Consequently $G^{-1} \in A'_1$ with
$$
f_1(x)= \sqrt { -2\log x -\log (-\log x)-2 \log \pi }\qquad \mbox{and} \qquad f_2(\alpha)=\frac 1 2 \log \alpha,
$$
implying $C_1=0$ and $C_2=\frac 1 2$.
\item {\bf Weibull distributions:} In this case, with $a\geq 0$ and $0< b\leq 1$, $G(x)=e^{-(x/\lambda)^k}$ with $\lambda>0$ and $k \in \N^*$, for $x\geq 0$. Then it is obvious that $G^{-1}(x) =\lambda \big ( - \log x \big )^{1/k}$ and therefore $G^{-1} \in A_1'$ with $f_1(x)=\lambda \big ( - \log x \big )^{1/k}$ and $f_2(\alpha)=\frac 1 k \, \log \alpha$ (implying $C_1=0$ and $C_2 =1/k$).

\end{itemize} 
\end{exemple}
When $G^{-1} \in A'_1$, it is possible to specify the limit distribution of  \eqref{asympto1}. Thus, we show the following result:
\begin{prop} \label{prop2}
Assume that $G^{-1} \in A'_1$. Then 
\begin{eqnarray}\label{asympto2}
\P \Big (\max_{j=1,\cdots,J} \{\tau_{n-j}'\} \leq x \Big )  \limiten \prod_{j=1}^{J} \big ( 1 - e^{-jx/C_2}\big ).
\end{eqnarray}
\end{prop} 
Such a result is interesting since it provides the asymptotic behavior of a vector of normalized and centered ratios $\tau_i$. Its asymptotic distribution is the distribution of a vector of independent exponentially distributed r.v's. However the parameters of these exponential distributions are different. Thus, if we consider the statistic  
\begin{eqnarray}\label{testT}
\widehat T= \max_{j=1,\cdots,J} \{\tau'_{n-j}\},
\end{eqnarray} 
the computation of the cumulative distribution function of $\widehat T$ requires consideration of the function $y \in [0,\infty) \mapsto R(y)=\prod_{j=1}^{J} \big ( 1 - e^{-jy}\big )$. This function converges quickly to $1$ when $J$ increases. Hence we numerically obtain that for $J\geq 3$, $R(3.042) \simeq 0.95$. Then we deduce from \eqref{asympto2} and \eqref{testT} that for $J\geq 3$, 
$$
\P \big ( \widehat T \leq  3.042 \times C_2  \big ) \simeq 0.95.
$$
This implies that for instance that for $J\geq 3$ and $n$ large enough,
\begin{eqnarray*}
&\bullet& \P\Big (\widehat T \leq 3.042\Big)\simeq 0.95\qquad \mbox{when  $X$ follows a Gamma distribution} \\
&\bullet& \P\Big(\widehat T \leq 1.521\Big)\simeq 0.95\qquad \mbox{when $|X|=|{\cal N}(0,1)|$,}
\end{eqnarray*}
with the computation of $C_2$  for each distribution. We remark that the ratio $\tau'_{n-1}$ is the main contributor to the statistic $\widehat T$ and it contains almost all the information. For giving  equivalent weights to the other ratios $\tau'_{k}$, $k \leq n-1$ and  so as not to be troubled by the nuisance parameter $C_2$, it is necessary to modify the statistic $\widehat T$. Then we consider: 
\begin{eqnarray}\label{testT2}
\widetilde T_n=  \max_{j=1,\cdots,J} \big \{j\, \tau'_{n-j} \big \}\times \frac 1 {\overline{s}_J}  \qquad \mbox{where}\qquad 
\overline{s}_J=\frac 1 {J} \, \sum_{j=1}^{J} j\,  \tau'_{n-j}. 
\end{eqnarray}
For $(u_n)_n$ and $(v_n)_n$ two sequences of real numbers, denote $u_n \simn v_n$ when $u_n/v_n \limiten 1$. The following proposition can be established:
\begin{prop}\label{prop3}
Assume that $G^{-1} \in A_1'$. Then, for a sequence  $(J_n)_n$ satisfying  $J_n \limiten \infty$ and $J_n/\log n \limiten 0$,
\begin{eqnarray}\label{asympto3}
\Pr \big ( \widetilde T_n \leq x \big )  \simn \big (1 -e^{-x} \big )^{J_n} \quad \mbox{for $x>0$}.
\end{eqnarray}
\end{prop}
\noindent In the case where $G^{-1} \in A_2$, similar results can also be  established, we demonstrated below. 
\begin{exemple}
Here there are some examples of classical distributions such as $G^{-1} \in A_2$:
\begin{itemize}
\item {\bf Pareto distribution ${\cal P}(\alpha)$:} In this case, with $c>0$ and $C>0$,  $G^{-1}(x)=C \, x^{-c}$ for $x \to 0$, and this implies $G^{-1} \in A_2$ with $a=c$. 
\item {\bf Burr distributions ${\cal B}(\alpha)$:} In this case, $G(x)=(1+x^c)^{-k}$ for $c$ and $k$ positive real numbers. Thus $G^{-1}(x)= (x^{-1/k}-1)^{1/c}$ for $x\in [0,1] $, implying $G^{-1} \in A_2$ with $a=(ck)^{-1}$. 
\item {\bf Absolute value of Student distribution $|t(\nu)|$ with $\nu$ degrees of freedom:} In the case of a Student distribution with $\nu$ degrees of freedom, the cumulative distribution function is $F_{t(\nu)}(x)=\frac 1 2 (1+I(y,\nu/2,1/2))$ with $y=\nu (\nu+x^2)^{-1}$ and therefore $G_{|t(\nu)|}(x)=I(y,\nu/2,1/2)$, where $I$ is the normalized beta incomplete function. Using the handbook of Abramowitz and Stegun (1964), we have the following expansion $ G_{|t(\nu)|}(x) =\frac {2 \nu ^{\nu/2-1}}{ B(\nu/2,1/2)} \, x ^{-\nu} + O(x^{-\nu +1})$
 for $x \to 0$, where $ B$ is the usual Beta function. Therefore, 
$$
G_{|t(\nu)|}^{-1}(x)=\frac { B(\nu/2,1/2)} {2 \nu ^{\nu/2-1}} \, x ^{-1/\nu} + O(x^{-1/\nu-1})\qquad x \to \infty.
$$
Consequently $G_{|t(\nu)|}^{-1} \in A_2$ with $a=1/\nu$.
\end{itemize} 
\end{exemple}
\begin{rem}
The case of standardized log-normal distribution is singular. Indeed, the probability distribution of $X$ is the same than the one of $\exp(Z)$ where $Z \sim {\cal N}(0,1)$. Therefore, $G(x)=\frac 1 2 \, \mbox{erfc}\big ( \frac {\log x}{\sqrt 2 } \big )$ implying $G^{-1}(x)=\exp \big ( \sqrt 2 \, \mbox{erfc}^{-1}(2x)\big )$. Using the previous expansion (\ref{erfc}), we obtain for any $\alpha>0$:
\begin{eqnarray*}
G^{-1}(\alpha \, x) &= &\exp \big ( \sqrt 2 \, \mbox{erfc}^{-1}(2x \, \alpha)\big ) \\
& = & \exp \Big ( \sqrt 2 \, \mbox{erfc}^{-1}( \, 2x) \big ( 1 + \frac {\log \alpha } {2 \log x } +  O(|(\ln x )^{-2}| )  \big )\Big ) \\
& =& G^{-1}( x)\big ( 1  +  O(|(\ln x )^{-1/2}| ). 
\end{eqnarray*}
Therefore, the standardized log-normal distribution is such that $G^{-1} \notin A_1 \cup A_2$. 
\end{rem}
\noindent For probability distributions such as $G^{-1} \in A_2$ we obtain the following classical result (see also Embrechts {\it et al.}, 1997): 
\begin{prop} \label{prop4}
Assume that $G^{-1} \in A_2$. Then, 
\begin{eqnarray}\label{asympto4}
\P \Big (\max_{j=1,\cdots,J}\{\log(\tau_{n-j})\} \leq x \Big )  \limiten \prod_{j=1}^{J} \big ( 1 - e^{-jx/a}\big ).
\end{eqnarray}
\end{prop} 
\noindent Hence the case of $G^{-1} \in A_2$ also provides interesting asymptotic properties on the ratios. In the forthcoming section devoted to the construction of an outlier detector from previous results, we are going to consider a test statistic which could be as well applied to distributions with functions $G^{-1}$ belonging to $A'_1$ and $A_2$. 
\section{A new non-parametric outlier detector} \label{stat}
We are going to consider the following test problem:
\begin{equation}\label{test}
\left \{ \begin{array}{l}H_0:\mbox{ there is no outlier in the sample}\\
H_1:\mbox{ there is at least one outlier in the sample} 
\end{array}
\right . .
\end{equation}
However we have to specify which kind of outlier and therefore which kind of contamination we consider. 
Our guide for this is typically the case of oversized regression residuals. Thus we would like to detect from $(X_1,\ldots,X_n)$ when there is a "gap" between numerous $X_i$ coming from a common probability distribution and one or several (but not a lot!) $X_j$ which are larger than the other one and generated from another distributions. As a consequence the previous test problem can be specified as follows:
\begin{equation}\label{test2}
\left \{ \begin{array}{l}H_0:\mbox{ $(X_1,\ldots,X_n)$ are i.i.d. r.v. with $G^{-1} \in A'_1 \cup A_2$}\\
H_1:\mbox{ there exists $K \in \N^*$ such as $(X_{(1)},\ldots,X_{(n-K)})$ are i.i.d. r.v. with $G^{-1} \in A'_1 \cup A_2$} \\
\qquad \qquad \qquad  \mbox{and for $i=0,\ldots,K-1$, the distribution  of $X_{(n-i)}$ satisfy $G^{-1}_{X_{(n-i)}}>G^{-1}$.}
\end{array}
\right . .
\end{equation}
A consequence of this specification is the following: we would like to detect outliers which appear as oversized data. Hence, under $H_1$ we could expect that there exists a "jump" between the smallest outlier and the largest non contaminated data. As a consequence under $H_1$ we could expect that the ratio $\tau_{(n-J)}$ is larger than it should be. \\
For doing such a job, we propose  to consider the following outlier detector based on ratios and which could be used as well when $G^{-1}$ belongs to $A'_1$ or $A_2$. Hence, define:
\begin{eqnarray}\label{D}
\widehat D_{J_n}=\frac {\log 2} {\widehat L_{J_n}} \,  \max _{j=1,\cdots,J_n} j \log ( \tau_{n-j} ) \qquad \mbox{where} \qquad {\widehat L_{J_n}}=\mbox{median}\big \{\big (j \log ( \tau_{n-j})  \big )_{1\leq j \leq J_n}\big \}.
\end{eqnarray}
Then, we obtain the following theorem:
\begin{theorem}\label{theo}
Assume that $G^{-1}\in A'_1 \cup A_2$. Then, for a sequence $(J_n)_n$ satisfying $J_n \limiten \infty$ and $J_n/\log n \limiten 0$,
\begin{eqnarray}\label{asympto5}
\Pr \big ( \widehat D_{J_n} \leq x \big )  \simn \big (1 -e^{-x} \big )^{J_n}.
\end{eqnarray}
\end{theorem}
\begin{rem}
In the definition of $\widehat D_{J_n}$ we prefer an estimation of the parameter of the exponential distribution with a robust estimator (median) instead of the usual efficient estimator (empirical mean), since several outliers could corrupt this estimation.
\end{rem}
\noindent Therefore, Theorem \ref{theo} can be applied for distributions with $G^{-1}$ belonging to $A'_1$ or $A_2$, {\it i.e.} as well as for Gaussian, Gamma or Pareto distributions. Hence, for a type I error $\alpha \in (0,1)$, the outlier detector $\widehat D_{J_n}$ can be computed, and with $t=-\log \big ( 1 -(1-\alpha)^{1/J_n} \big )$,
\begin{itemize}
\item If  $\widehat D_{J_n} \leq t$ then we conclude that there is no outlier in the sample. 
\item If $\widehat D_{J_n} > t$  then the largest index $\widehat k_0$ such as $\widehat k_0 \, \log (\tau_{n-\widehat k_0})/\widehat L_{J_n} \geq t$ indicates that we decide that the observed data  $(X_{(n-\widehat k_0+1)},X_{(n-\widehat k_0+2)},\ldots, X_{(n)})$ can be considered as outliers, implying that there are $\widehat k_0$ detected outliers.
\end{itemize}
As a consequence this outlier detector allows a decision of the test problem \eqref{test} and also the identification of the exact outliers. 
\section{Monte-Carlo experiments}\label{simu}
We are going to compare the new outlier detector defined in (\ref{D}) with usual univariate outlier detectors. After giving some practical details of the application of $\widehat D_{J_n}$, we present the results of Monte-Carlo experiments under several probability distributions. 
\subsubsection*{Practical procedures of outlier detections}
The definition of $\widehat D_{J_n}$ is simple,
and in practice just requires the specification of $2$ parameters: 
\begin{itemize}
\item The type I error $\alpha$ is the risk to detect outliers in the sample while there is no outlier. Hence, a natural choice could be the "canonical" $\alpha=0.05$. However, we chose to be strict concerning the risk of false detection, {\it i.e.} we chose $\alpha=0.007$ (as it was chosen by Tukey himself for building boxplots) which implies that we prefer not to detect "small" outliers and hence we avoid to detect a large number of outliers while there is no outlier.
\item The number  $J_n$ of considered ratios. On the one hand, it is clear that the smaller $J_n$, the smaller the detection threshold, therefore more sensitive is the detector to the presence of outliers. On the other hand, the larger $J_n$, the more precise is the estimation of the parameter of asymptotic exponential distribution (the convergence rate of $\widehat L_{J_n}$ is $\sqrt n$) and larger is the possible number of detected outliers.
We carried out numerical simulations using $20000$ independent replications, for several probability distributions (the seven distributions presented below) for several values of the number of outliers $K$, sample size $n$ and parameter $J$. Results are reported in Table \ref{Table00}. A first conclusion: the larger $n$ and $K$ the larger test power. Another  conclusion, but this is not a surprise, is the fact that the "optimal" choice of $J$ depends on $K$ and $n$. \\
As a consequence, for at least detecting $K=10$ outliers, we use $J_n=1+[4*\log^{3/4}(n)]$ that is an arbitrary choice satisfying $J_n=o(\log(n))$ and fitting well the results of these simulations, {\it i.e.} for $n=100$, $J_n=13$, for $n=1000$, $J_n=18$ and for $n=5000$, $J_n=20$. 
\end{itemize}
\begin{table}
\caption{ \label{Table00} Numerical choice of $J_n$: average frequencies for potential outliers with $\widehat D_J$ when there are $K$ outliers from the seven considered probability distributions, and for several values of $K$, sample size $n$ and parameter $J$. Here $\alpha=0.007$, $20000$ independent replications are done and multiplicative outliers (see below) are generated.}
\begin{center}
\begin{tabular}{|l|c|c|c|c|c|c|c|c|c|}
$\qquad \qquad J$ & 12 & 14 & 16 & 18 & 20 & 22 & 24 & 26 & 28 \\
\hline\hline 
$n=100$, $K=5$ &  0.595 & 0.588 & 0.583 & 0.573 & 0.568 & 0.565& 0.561 &0.557 & 0.554   \\
$\qquad \qquad K=10$ & 0.694 & 0.708 & 0.702 & 0.696 & 0.694 & 0.689 & 0.686&  0.683& 0.679 \\  \hline 
$n=500$, $K=5$ &0.727 & 0.719& 0.712 &0.705&0.700 &0.691& 0.687& 0.685& 0.680  \\  
$\qquad \qquad K=10$ &  0.796 & 0.818 & 0.823 &  0.822 & 0.817 & 0.813 & 0.811 & 0.810 &0.804 \\ \hline
$n=1000$, $K=5$ & 0.760& 0.752 & 0.744&  0.738 &  0.732&  0.727 & 0.722 &  0.716 &0.711  \\
$\qquad \qquad K=10$ & 0.822 & 0.848 &0.857 &0.858 & 0.856 &0.853 &  0.850 & 0.847 & 0.844 \\
\hline
$n=5000$, $K=5$ & 0.802 & 0.793 &0.786 & 0.779 & 0.773 & 0.769 & 0.763 & 0.759 & 0.754 \\
$\qquad \qquad K=10$ & 0.846 & 0.878 & 0.891 & 0.894 & 0.896 &  0.893 & 0.890 & 0.888 & 0.887\\
\hline
\end{tabular}
\end{center}
\end{table}
We have compared the new detector $\widehat D_{J_n}$ to five common and well-known univariate outlier detectors computed from the sample $(X_1,\cdots,X_n)$. 
\begin{enumerate}
\item The Student's detector (see for instance Rousseeuw and Leroy, 2005): an observation from the sample $(X_1,\cdots,X_n)$ will be consider as an outlier when $\P(X_k>\overline X_n+s_s \times  \overline \sigma_n)$ where $\overline X_n$ and $\overline \sigma_n^2$ are respectively the usual empirical mean and variance computed from $(X_1,\cdots,X_n)$, and $s_s$ is a threshold. This threshold is usually computed from the assumption that $(X_1,\cdots,X_n)$ is a Gaussian sample and therefore  $s_s=q_{t(n-1)}\big ((1-\alpha/2)\big )$, where $q_{t(n-1)}(p)$ denotes the quantile of the Student distribution with $(n-1)$ degrees of freedom for a probability $p$.
\item The Tukey's detector (see Tukey, 1977) which provides the famous and usual boxplots:  $X_k$ is considered to be an outlier from $(X_1,\cdots,X_n)$ if $X_k>Q3+1.5 \times IQR$, where $IQR=Q3-Q1$, with $Q_3$ and $Q_1$ the third and first empirical quartiles of $(X_1,\cdots,X_n)$. Note that the confidence factor $1.5$ was chosen by Tukey such that the probability of Gaussian random variable to be decided as an outlier is close $0.7 \%$ which is good trade-off.
\item An adjusted Tukey's detector as it was introduced and studied in Hubert and Vandervieren (2008):  $X_k$ is considered to be an outlier from $(X_1,\cdots,X_n)$ if $X_k>Q3+1.5 Q3 + 1.5\, e^{3 \, MC} \times IQR$, where $MC$ is the medcouple, defined by $MC = \mbox{median}_{X_i\leq Q2\leq X_j} h(X_i,X_j)$, where $Q2$ is the sample median and the kernel function $h$ is given by
$h(X_i,X_j)= \frac { (X_j-Q2)-(Q2-X_i )}{X_j-X_i}$. This new outlier detector improves considerably the accuracy of the usual Tukey's detector for skewed distributions.
\item The $MAD_e$ detector (see for instance Rousseeuw and Leroy, 2005): $X_k$ is considered as an outlier from $(X_1,\cdots,X_n)$ if $|X_k-Q2| > 3*1.483*\mbox{median}(|X_1-Q2|,\cdots, |X_n-Q2|)$. The coefficient $1.483$ is obtained from the Gaussian case based on the relation $SD\simeq 1.483 \times \mbox{median}(|X_1-Q2|,\cdots, |X_n-Q2|)$, while the confidence factor $c=3$ is selected for building a conservative test.
\item The Local Outlier Factor (LOF), which is a non-parametric detector (see for instance Breunig {\it et al.}, 2000). This procedure is based on this principle: an outlier can be distinguished when its normalized density (see its definition in Breunig {\it et al.}, 2000) is larger than $1$ or than a threshold larger than $1$. However, the computation of this density requires to fix the parameter $k$ of the used $k$-distance and a procedure or a theory for choosing a priori $k$ does not still exist. The authors recommend $k>10$ and $k<50$ (generally). We chose to fix  $k=J_n$, where $J_n$ is used for the computation of $\widehat D_J$. Then, using the same kind of simulations than those reported in Table \ref{Table00}, we tried to optimize the choice of a threshold $s_{LOF}$ defined by: if $LOF(X_i)>s_{LOF}$ then the observation $X_i$ is considered to be an outlier. We remark that it is not really possible to choose a priori $k$ and $s_{LOF}$ with respect to $\alpha$. Table \ref{Table01} provides the results of simulations from $10000$ independent replications and the seven probability distributions. We have chosen to optimize a sum of empirical type I (case $K=0$) and II (case $K=5$ and $K=10$) errors. This leads one to choose $s_{LOF}=8$ for $n=100$ as well as for $n=1000$. 
\end{enumerate}
\begin{table}
\caption{ \label{Table01} Numerical choice of the LOF threshold $s_{LOF}$: Average frequencies of potential outliers with LOF detector when there are $K$ outliers from the seven considered probability distributions, and for several values of $K$, sample size $n$ and threshold $s_{LOF}$. Note that $10000$ independent replications are generated and multiplicative outliers (see below) are generated.}
\begin{center}
\begin{tabular}{|l|l|c|c|c|c|c|c|}
$s_{LOF}$ & $K$ &2 & 4 & 6 & 8 & 10 & 12  \\
\hline\hline 
$n=100$&$K=0$ &  0.903 & 0.447 & 0.285 & 0.222 & 0.174 & 0.150  \\
&$K=5$ & 1.000 & 1.000 & 1.000 & 0.996 & 0.961 & 0.863  \\ 
 &$K=10$ & 1.000 & 0.860 & 0.422 & 0.232 & 0.141 & 0.106  \\ \hline 
$n=1000$& $K=0$ &  0.955 & 0.533 & 0.328 & 0.237 & 0.196 & 0.179  \\
&$K=5$ & 1.000 & 1.000 & 1.000 & 1.000 & 1.000 & 0.995  \\ 
 &$K=10$ & 1.000 & 1.000 & 0.981 & 0.888 & 0.744 & 0.581  \\ \hline 
\hline
\end{tabular}
\end{center}
\end{table}
\noindent Student, Tukey and $MAD_e$ detectors are more or less based on Gaussian computations. We would not be surprised if those methods failed to detect outliers when the distribution of $X$ is "far" from the Gaussian distribution (but these usual detections of outliers, for instance the Student detection obtained on Studentized residuals from a least squares regression, are done even if the Gaussian distribution is not attested). Moreover, the computations of these detectors' thresholds are based on an individual detection of outlier, {\it i.e.} a test deciding if a fixed observation $X_{i_0}$ is an outlier or not. Hence, if we apply them to each observation of the sample, the probability to detect an outlier increases with $n$. This is not exactly the same as deciding whether  if there are no outliers in a sample, which is the object of test problem \eqref{test}.\\
Then, to compare these detectors to $\widehat D_{J_n}$, it is appropriate to change the thresholds of these detectors following test problem \eqref{test} and its precision \eqref{test2}. Hence we have to define a threshold $s>0$ of this test from the relation $\P(\exists k=1,\cdots,n, ~X_k>s)=\alpha$. Therefore, from the independence property $1-\alpha=\P(X_1<s)^n$ implying that $s$ has to satisfy $\P(X_1>s)=(1-(1-\alpha)^{1/n}) \simeq \alpha/n$ when $n$ is large and $\alpha$ close to $0$ (typically $\alpha=0.01$). In the sequel we are going to compute the confidence factors of previous famous detectors to the particular case of absolute values of  Gaussian variables. Hence, when for $Z$ a ${\cal N}(0,1)$ random variable and if $n$ is large, then $\P(|Z|>s_G)\simeq \frac \alpha n$ with $s_G=q_{{\cal N}(0,1)}(1 -\frac \alpha {2n}) $. 
Then, we define:
\begin{enumerate}
\item The Student detector $2$: we consider that $X_k$ from $(X_1,\cdots,X_n)$ is an outlier when  $X_k>\overline X_n+c_S \times\overline \sigma_n$,  with $2/\sqrt{2\pi}+ c_S \sqrt{ (\pi-2)/\pi }=s_G$ implying $c_S=\big ( q_{{\cal N}(0,1)}(1 -\frac \alpha {2n})-\sqrt{2/\pi} \big )\sqrt{\pi/ (\pi-2)}$ (here we assume that $n$ is a large number inducing that the Student distribution with $(n-1)$ degrees of freedom could be approximated by the standard Gaussian distribution).
\item The Tukey detector $2$: we consider that $X_k$ from $(X_1,\cdots,X_n)$ is an outlier when $X_k>Q3+c_T\times IQR$. In the case of the absolute value of a standard Gaussian variable, $Q1\simeq 0.32$ and $Q3 \simeq 1.15 $ implying $c_T=(s_G-1.15)/0.83\simeq 1.20 \times s_G-1.38$. 
\item The $MAD_e$ detector $2$: we consider that $X_k$ from $(X_1,\cdots,X_n)$ is an outlier when $X_k-Q2 > c_M \times \mbox{median}(|X_1-Q2|,\cdots, |X_n-Q2|)$. In the case of the absolute value of a standard Gaussian variable, $Q2\simeq 0.67$ and $\mbox{median}(|X_1-Q2|,\cdots, |X_n-Q2|)\simeq 0.40$. This induces $c_M\simeq 2.50 \times s_G-1.69$. 
\end{enumerate}
\subsubsection*{First results of Monte-Carlo experiments for samples without outlier: the size of the test}
We apply the different detectors in different frames and for several probability distributions which are:
\begin{itemize}
\item The absolute value of Gaussian distribution with expectation $0$ and variance $1$, denoted $\big |{\cal N}(0,1)\big |$  ({\it case $A_1'$});
\item The exponential distribution with parameter $1$, denoted ${\cal E}(1)$ ({\it case $A_1'$});
\item The Gamma distribution  with parameter $3$, denoted $\Gamma(3)$ ({\it case $A'_1$});
\item The Weibull distribution with parameters $(3,4)$, denoted $W(3,4)$ ({\it case $A'_1$});
\item The absolute value of a Student distribution with $2$ degrees of freedom, denoted $|t(2)|$ ({\it case $A_2$});
\item The standard log-normal distribution, denoted $\log- {\cal N}(0,1)$ ({\it not case $A'_1$ or $A_2$});
\item The absolute value of a Cauchy distribution, denoted $|{\cal C}|$ ({\it case $A_2$}).
\end{itemize}
In the sequel, we will consider samples $(X_1,\cdots,X_n)$ following these probability distributions, for $n=100$ and $n=1000$, and for several numbers of outliers.  \\
We begin by generating independent replications of samples without outlier, which corresponds to be under $H_0$. Then we apply the outlier detectors. The results are reported in Table \ref{Table5}. \\
\begin{table}
\caption{ \label{Table5} The size of the test: average frequencies (Av. Freq.) of potential outliers with the different outlier detectors, for the different probability distributions, $n=100$ and $n=1000$, while there is no generated outlier in samples. Here $\alpha=0.007$ and $20000$ independent replications are generated.}
\begin{center}
\begin{tabular}{|l|c|c|c|c|c|c|c|}
\hline\hline
 $n=100$  & $\big | {\cal N}(0,1) \big |$& ${\cal E}(1)$ & $\Gamma(3)$ &  $W(3,4)$ & $| t(2) |$ & $\log- {\cal N}(0,1)$  & $|{\cal C}|$\\
\hline \hline
Av. Freq. $\widetilde D_{J_n}$   & 0.007   &  0.008 &    0.008&     0.008 &    0.010 &  0.010 & 0.018 \\
\hline 
Av. Freq. LOF   & 0.001   &  0.029 &    0.013&     0 &  0.643 &   0.259   &  0.970\\
\hline 
Av. Freq.  Student   & 0.861   &  0.993 &   0.921&0.308 &     1 &    1  &  1  \\
\hline 
Av. Freq. Tukey   & 0.804   &  0.991 &0.918 &  0.327  &   1 &    1  &  1   \\
\hline 
Av. Freq. Adj. Tukey   & 0.213   &  0.344 &0.408 &  0.353 &     0.844&    0.710    &  0.981   \\
\hline 
Av. Freq. $MAD_e$   & 0.751   &  0.995 & 0.879 &  0.163  &    1 &    1  &  1   \\
\hline 
Av. Freq. Student $2$   & 0.002 &  0.110 &  0.017 &  0 &     0.655 &    0.515  &  0.936 \\
\hline 
Av. Freq. Tukey $2$  & 0.022  &  0.486 &    0.119& 0 &  0.951&    0.937    & 1 \\
\hline 
Av. Freq. $MAD_e$ $2$  & 0.023   &  0.624 &    0.112&   0 &  0.971 &    0.975  &  1   \\
\hline\hline
 $n=1000$  & $\big | {\cal N}(0,1) \big |$& ${\cal E}(1)$ & $\Gamma(3)$ &  $W(3,4)$ & $| t(2) |$ & $\log- {\cal N}(0,1)$  & $|{\cal C}|$\\
\hline \hline
Av. Freq. $\widetilde D_{J_n}$   & 0.009   &  0.009 &    0.009&     0.009 &   0.014 & 0.011   & 0.016 \\
\hline 
Av. Freq. LOF   & 0.005   &  0.023 &    0.019&     0.001 &    0.843&  0.281    &  0.998\\
\hline 
Av. Freq. Student   & 1   &  1 &   1&0.986 &    1 &   1  &  1  \\
\hline 
Av. Freq. Tukey   & 1  &  1 &1 &  0.939   &  1 &    1  &  1   \\
\hline 
Av. Freq. Adj Tukey   & 0.492  &  0.899 &0.950 &  0.796   &  1 &    1  &  1   \\
\hline 
Av. Freq. $MAD_e$   & 1   &  1 & 1 &  0.664  & 1 &    1  &  1   \\
\hline 
Av. Freq. Student $2$   & 0.006   &  0.552 &0.112 &  0  & 1  &    0.998&  1 \\
\hline 
Av. Freq. Tukey $2$  & 0.008  &  0.943 &    0.259& 0 &  1 &    1  & 1 \\
\hline 
Av. Freq. $MAD_e$ $2$  & 0.008   &  0.991 &    0.246&   0&  1 &    1 &  1  \\
\hline
\hline
\end{tabular}
\end{center}
\end{table}
~\\
A first conclusion from this simulations is the following: as we already said, the original Student, Tukey and $MAD_e$ detectors 
can not be compared to the detector $\widetilde D_{J_n}$ because their empirical sizes came out larger: they are not constructed to answer to our test problem. Therefore we are going now to consider only their global versions Student $2$, Tukey $2$ and $MAD_e$ $2$. Moreover, as they have been constructed from the Gaussian case, these second versions of detectors provide generally poor results in case of non-Gaussian distributions especially for Student, log-normal and Cauchy distributions (where outliers are always detected while there are no generated outliers). 
\subsubsection*{Second results of Monte-Carlo experiments for samples with outliers}
Now, we consider the cases where there is a few number of outliers in the samples $(X_1,\cdots,X_n)$. 
Denote $K$ the number of outliers, and $\ell>0$ a real number which represents a positive real number. We generated $3$ kinds of contaminations:
\begin{itemize}
\item A shift contamination: $(X_{(1)} ,\cdots,X_{(n-K)}, X_{(n-K+1)}+\ell, \cdots,X_{(n)}+\ell)$ instead of $(X_{(1)} ,\cdots,X_{(n)})$. We chose $\ell=10$. By the way the cluster of outliers is necessarily separated from the cluster of non-outliers.
\item A multiplicative contamination: $(X_{(1)} ,\cdots,X_{(n-K)}, 3  X_{(n-K+1)},\cdots, 3 X_{(n)})$ instead of $(X_{(1)} ,\cdots,X_{(n)})$. By the way the cluster of outliers is necessarily separated from the cluster of non-outliers.
\item A point contamination: $(X_{(1)} ,\cdots,X_{(n-K)}, \ell,\cdots, \ell)$ instead of $(X_{(1)} ,\cdots,X_{(n)})$. We chose $\ell=1000$. By the way, the cluster of outliers is generally separated to the cluster of non-outliers (but not necessary, especially for Student or Cauchy probability distributions).
\end{itemize}  
First, we consider the second versions of Student, Tukey and $MAD_e$ detectors. But we also consider parametric versions of these detectors, which are denoted Student-para, Tukey-para and $MAD_e$-para: the thresholds and confidence  factors are computed and used with the knowledge of the probability distribution of the sample (these thresholds change following the considered probability distributions). Hence, they are parametric detectors while $\widetilde D_{J_n}$ or LOF are non-parametric detectors. We chose these parametric versions because they allow to obtain the same size of all the detectors and then a comparison of the test powers is more significant. \\
~\\  
First results are reported in Tables \ref{Table8} ($n=100$) and \ref{Table81} ($n=1000$)  for $K$ shifted outliers ($K=5$ or $K=10$). From Tables \ref{Table8} and \ref{Table81}, it appears:
\begin{itemize}
\item Student $2$, as well as Student-para detectors are not really good choices for the detection of outliers because they are not robust statistics (the empirical variance is totally modified by the values of outliers).
\item Tukey and $MAD_e$ detectors provide more and less similar results. But even their parametric versions are not able to detect outliers for skewed distributions (Student, log-normal and Cauchy distributions). With the same size,  $\widetilde D_{J_n}$ clearly provides better results, even if they are not very accurate (especially for the Cauchy distribution). 
\item LOF detector does not provide accurate results (especially when $K=10$). It is certainly a more interesting alternative in case of multivariate data.
\end{itemize} 
\begin{table}
\caption{ \label{Table8} Average frequencies (Av. Freq.) of potential outliers with the different outlier detectors and average number (Av. Numb.) of detected outliers when outliers are detected, for $n=100$, with $K$ shifted outliers. Here $\alpha=0.007$ and $20000$ independent replications are generated.} 
\begin{center}
\footnotesize \begin{tabular}{|l|c||c|c|c|c|c|c|c|}
\hline\hline
 $n=100$ &$K$ & $\big | {\cal N}(0,1) \big |$& ${\cal E}(1)$ & $\Gamma(3)$ &  $W(3,4)$ &  $| t(2) |$ &$\log- {\cal N}(0,1)$  & $|{\cal C}|$\\
\hline \hline
Av. Freq. $\widetilde D_{J_n}$ & 0 & 0.007 &  0.008 &  0.008& 0.008 & 0.010 & 0.010&   0.018  \\
& 5 & 0.998  &  0.562 &  0.574& 1 & 0.848 & 0.813&   0.066 \\
& 10 & 1  & 0.973 & 0.959& 1 & 1 & 1&  0.076 \\
Av. Numb. & 5 & 5.12 &  5.23 &  5.25& 5.14 & 5.09 & 5.07 & 9.17 \\
 & 10 & 10.60  & 10.64 & 10.58& 10.62 & 10.05 & 10.05& 10.01\\
\hline 
Av. Freq. Student 2 & 0 & 0.001 &  0.111 &  0.018& 0 & 0.655 & 0.515& 0.934  \\
& 5 & 0  &  0.013 &  0.004& 0 & 0.538& 0.411& 0.687 \\
& 10 & 0 & 0 & 0& 0 & 0.241 & 0.126 & 0.628\\
Av. Numb. & 5 & 0  &  1 &  1& 0 & 1.04 & 1.02 & 1.67 \\
 & 10 & 0  & 0 &  0& 0 &1.01  & 1 & 1.66\\
\hline 
Av. Freq. Tukey 2 & 0 & 0.022 &  0.483 &  0.118& 0 & 0.951 & 0.937&  1  \\
& 5 & 1  &  1 &  1& 1 & 1 & 1 & 1  \\
& 10 & 1  &  1 &  0.999& 1 & 1 & 1 & 1 \\
Av. Numb. & 5 & 5.01   &  5.35&  5.05& 5& 5.10 & 5.09 & 10.32\\
 & 10 & 10   & 10.14 & 9.67& 10 & 10 & 10 & 13.19\\
\hline 
Av. Freq. $MAD_e$ 2 & 0 & 0.023  &  0.619 &  0.109& 0 & 0.971 & 0.975&  1  \\
& 5 & 1  &  1 &  1& 1 & 1 & 1 &  1  \\
& 10 & 1  &  1 &  0.999& 1 & 1 & 1 & 1 \\
Av. Numb. & 5 & 5.01   &  5.67&  5.06& 5 & 5.24 & 5.37& 12.66\\
 & 10 & 10  &  10.38 & 9.86& 10 & 10 & 12.16 & 16.43\\
\hline 
Av. Freq. LOF & 0 & 0.001   &  0.029 &    0.013&     0 &    0.643  &  0.259 &  0.970 \\
& 5 & 0.998  &  0.521 &  0.123& 1 & 0.309 & 0.129 &  1  \\
& 10 & 0.046  &  0.011 &  0& 0.200 & 0.03 & 0.003 & 0.548 \\
Av. Numb. & 5 & 5  &  3.97&  2.86& 5 & 2.71& 2.03& 3.39\\
 & 10 & 9.11  &  1.10 & - & 9.47 & 1.56 & 1.40 & 2.27\\
\hline 
Av. Freq. Student-para & 0 & 0.002 &  0 &  0.001& 0.005 & - & 0& -  \\
& 5 & 0  &  0 &  0& 1 & - & 0& - \\
& 10 & 0 & 0 & 0& 0 & - & 0  &  -  \\
Av. Numb. & 5 & 0&  0 & 0& 5 & -  & 0 & - \\
 & 10 & 0  & 0 &  0 & 0 & -  & 0 & -\\
\hline 
Av. Freq. Tukey-para & 0 & 0.022 &  0.013 &  0.014& 0.025 & 0.008 & 0.009&   0.007  \\
& 5 & 1  &  0.991 &  0.960& 1 & 0.009 & 0.026&   0.005  \\
& 10 & 1  &  0.903 &  0.782& 1 & 0.007& 0.026&   0.004  \\
Av. Numb. & 5 & 5.01   &  4.87&  4.40& 5.01& 1 & 1.03 & 1.94\\
 & 10 & 10   & 7.27 & 4.50& 10 & 1 & 1.01 & 1.88\\
\hline 
Av. Freq. $MAD_e$-para & 0 & 0.023  &  0.014 &  0.013& 0.025 & 0.007 & 0.009&   0.007  \\
& 5 & 1  &  0.995 &  0.969& 1 & 0.009& 0.024&  0.006  \\
& 10 & 1  &  0.966 &  0.867& 1 & 0.008 & 0.025& 0.005 \\
Av. Numb. & 5 & 5.01   &  4.92 & 4.47& 5.01 & 1 & 1.02& 1.94\\
 & 10 & 10  &  8.51 & 5.46& 10 & 1 & 1.02 & 1.89\\
\hline 
\hline
\end{tabular}
\end{center}
\end{table}
\begin{table}
\caption{ \label{Table81} Average frequencies (Av. Freq.) of potential outliers with the different outlier detectors and average number (Av. Numb.) of detected outliers when outliers are detected, for $n=1000$, with $K$ shifted outliers. Here $\alpha=0.007$ and $10000$ independent replications are generated.} 
\begin{center}
\footnotesize \begin{tabular}{|l|c||c|c|c|c|c|c|c|}
\hline\hline
 $n=1000$ &$K$ & $\big | {\cal N}(0,1) \big |$& ${\cal E}(1)$ & $\Gamma(3)$ &  $W(3,4)$& $| t(2) |$ & $\log- {\cal N}(0,1)$  & $|{\cal C}|$\\
\hline \hline
Av. Freq. $\widetilde D_{J_n}$ & 0 & 0.009 &  0.009 &  0.009& 0.009 & 0.014 & 0.011&   0.016  \\
& 5 & 1  &  0.892&  1& 1 & 0.241 & 0.517&   0.072  \\
& 10 & 1  &  0.997 &  0.991& 1 & 0.987 & 1&   0.088 \\
Av. Numb. & 5 & 5.23 &  5.27 &  5.32& 5.27 & 5.34 & 5.14& 8.85\\
 & 10 & 10.37  &  10.82 & 10.75& 10.41 & 10.05 & 10.06 & 7.83\\
\hline 
Av. Freq. Student 2 & 0 & 0.006  &  0.565 &  0.111& 0 & 1 & 0.998& 1 \\
& 5 & 1  &  1 &  1& 1 & 1& 1&  1 \\
& 10 & 1  & 1 &  1& 1 & 1& 1&  1 \\
Av. Numb. & 5 & 5  & 5.02&  5.01& 5 & 4.69 & 4.98 & 4.74\\
 & 10 & 10  & 10  & 9.97& 10& 7.73 & 9.75 & 4.59 \\
\hline 
Av. Freq. Tukey 2& 0 & 0.008  &  0.949&  0.256& 0 & 1 & 1&   1 \\
& 5 & 1  &  1 &  1& 1 & 1 & 1&  1 \\
& 10 & 1  &  1 &  1& 1 & 1 & 1& 1 \\
Av. Numb. & 5 & 5.01 & 7.90&  5.27& 5 & 22.34 & 20.09& 64.41\\
 & 10 & 10.01  & 12.64 & 10.23& 10.01 & 22.49 & 20.06 & 67.76\\
\hline 
Av. Freq. $MAD_e$ 2 & 0 & 0.008  &  0.992 &  0.237& 0& 1 & 1&  1 \\
& 5 & 1  &  1 &  1& 1 & 1 & 1 & 1 \\
& 10 & 1  &  1 &  1& 1 & 1& 1&  1 \\
Av. Numb. & 5 & 5.01   &  9.86 &  5.26& 5 & 26.94& 28.52 & 83.04\\
 & 10 & 10.01 & 14.55& 10.23& 10.01& 27.10 & 28.46& 86.59 \\
\hline
Av. Freq. LOF & 0 & 0.001   &  0.029 &    0.013&     0 &    0.843  &  0.281 &  0.970 \\
& 5 & 1  &  0.802 & 0.473& 1 & 0.291 & 0.113& 0.65  \\
& 10 & 0.961 &  0.014 &  0.004& 1 & 0.125 & 0.011 & 0.571 \\
Av. Numb. & 5 & 5.01  & 4.27&  3.37& 5 &2.21 &1.21 & 3.07\\
 & 10 & 9.93  & 3.36 & 3.50& 10 & 1.96 & 1.63 & 2.92\\
\hline 
Av. Freq. Student-para & 0 & 0.006  &  0.004 &  0.005& 0.007 & - & 0& - \\
& 5 & 1  &  0.652 &  0.979& 1 & - & 0&  - \\
& 10 & 1  & 0.017 &  0.275& 1 & - & 0&  - \\
Av. Numb. & 5 & 5  &1.21&  2.48& 5.01 & -  & 0 & - \\
 & 10 & 10  & 1  & 1.07& 10& - & 0 & - \\
\hline 
Av. Freq. Tukey-para & 0 & 0.008  &  0.008 &  0.008& 0.010 & 0.007 & 0.007&   0.007  \\
& 5 & 1  &  1 &  1& 1 & 0.007 & 0.013&   0.007  \\
& 10 & 1  &  1 &  1& 1 & 0.007 & 0.013&   0.007 \\
Av. Numb. & 5 & 5.01 &  5.01&  5.01& 1.02 & 1.95 & 1& 2.00\\
 & 10 & 10.01  & 10.01 & 9.93& 10.01 & 1& 1 & 2\\
\hline 
Av. Freq. $MAD_e$-para & 0 & 0.008  &  0.007 &  0.008& 0.010 & 0.007 & 0.007&   0.007 \\
& 5 & 1  &  1 &  1& 1 & 0.007& 0.013&   0.037  \\
& 10 & 1  &  1 &  1& 1 & 0.008 & 0.013&   0.007 \\
Av. Numb. & 5 & 5.01   &  5.01 &  5.01& 5 & 1.01 & 1.01 & 2\\
 & 10 & 10.01 & 10.01& 9.94& 10.01& 1  & 1& 2 \\
\hline 
\hline
\end{tabular}
\end{center}
\end{table}
Those conclusions are confirmed when both the other contaminations (multiplicative and point) are used for generating outliers (see Table \ref{Table7} and \ref{Table9}). 
 \begin{table}
\caption{ \label{Table7} Average frequencies (Av. Freq.) of potential outliers with the different outlier detectors and average number (Av. Numb.) of detected outliers when outliers are detected for $n=100$ and $n=1000$, with $K$ multiplicative outliers. Here $\alpha=0.007$ and $20000$ independent replications are generated.} 
\begin{center}
\footnotesize \begin{tabular}{|l|c||c|c|c|c|c|c|c|}
\hline\hline
 $n=100$ &$K$ & $\big | {\cal N}(0,1) \big |$& ${\cal E}(1)$ & $\Gamma(3)$ &  $W(3,4)$ &  $| t(2) |$ &$\log- {\cal N}(0,1)$ & $|{\cal C}|$\\
\hline \hline
Av. Freq. $\widetilde D_{J_n}$ & 0 & 0.007 &  0.008 &  0.008& 0.008 & 0.010& 0.010&   0.018  \\
& 5 & 1  &  0.620 &  1& 1 & 0.743 & 0.797&   0.221 \\
& 10 & 1  & 1 & 1& 1 & 1& 1&  0.910 \\
Av. Numb. & 5 & 5.03  &  5.04 &  5.03& 5.03 & 5.05  & 5.03 & 5.17\\
 & 10 & 10.01  & 10.01 & 10.01& 10.01& 10 & 10 & 10\\
\hline 
Av. Freq. Student-para & 0 & 0.002 &  0 &  0.001& 0.005 & - & 0& -  \\
& 5 & 0.022  &  0 &  0.021& 1 & - & 0&   -  \\
& 10 & 0.019 & 0 & 0.022& 0.190 & - & 0  &  -  \\
Av. Numb. & 5 & 1 &  0 &  1& 4.08 & -  & 0 & - \\
 & 10 & 1   & 0 &  1 & 1.04 & -  & 0 & -\\
\hline 
Av. Freq. Tukey-para & 0 & 0.022 &  0.013 &  0.014& 0.025 & 0.008 & 0.009&   0.007  \\
& 5 & 1  &  0.983 &  1& 1 & 0.065 & 0.304&   0.024  \\
& 10 & 1  &  0.980 &  1& 1 & 0.063 & 0.305&   0.022  \\
Av. Numb. & 5 & 5  &  3.83&  5& 5& 1.03 & 1.21 & 1\\
 & 10 & 9.99  & 4.45& 9.84& 10 & 1.03 & 1.19 & 1.01\\
\hline 
Av. Freq. $MAD_e$-para & 0 & 0.023  &  0.014 &  0.013& 0.025 & 0.007 & 0.009&   0.007  \\
& 5 & 1  &  0.981 &  1& 1 & 0.065& 0.303&  0.023  \\
& 10 & 1  & 0.977&  1& 1 & 0.063 & 0.303& 0.022  \\
Av. Numb. & 5 & 5  & 3.81  &  5& 5 & 1.03 & 1.21& 1\\
 & 10 & 9.99  & 4.48 &  9.77& 10 & 1.03 & 1.20 & 1.01\\
\hline 
\hline
 $n=1000$ &$K$ & $\big | {\cal N}(0,1) \big |$& ${\cal E}(1)$ & $\Gamma(3)$ &  $W(3,4)$ & $| t(2) |$ & $\log- {\cal N}(0,1)$  & $|{\cal C}|$\\
\hline \hline
Av. Freq. $\widetilde D_{J_n}$ & 0 & 0.009 &  0.009 &  0.009& 0.009 & 0.014 & 0.011&   0.016  \\
& 5 & 1  &  1&  1& 1 & 0.933 & 0.997&   0.200  \\
& 10 & 1  &  1 &  1& 1 & 1 & 1&   0.947 \\
Av. Numb. & 5 & 5.06  &  5.06 &  5.06& 5.06 & 5.08& 5.06& 5.34\\
 & 10 & 10.02 &  10.03 & 10.03& 10.03 & 10.03 & 10.03 & 10.02\\
\hline 
Av. Freq. Student-para & 0 & 0.006  &  0.004 &  0.005& 0.007 & - & -& - \\
& 5 & 1  &  0.996 &  1& 1 & - & 0 &  - \\
& 10 & 1  & 0.664 &  1& 1 & - &  0 &  - \\
Av. Numb. & 5 & 5  & 2.46 &  5& 5 & -  & 0 & - \\
 & 10 & 10  & 1.25  & 5.10& 10& - & 0 & - \\
\hline 
Av. Freq. Tukey-para & 0 & 0.008  &  0.008 &  0.008& 0.010 & 0.007 & 0.007&   0.007  \\
& 5 & 1  &  1 &  1& 1 & 0.061 & 0.441&   0.021  \\
& 10 & 1  &  1 &  1& 1 & 0.061 & 0.443&   0.021 \\
Av. Numb. & 5 & 5 &  5&  5& 5 & 1.02 & 1.32& 1.01\\
 & 10 & 10  & 9.97 & 10& 10 & 1.03 & 1.34 & 1.01\\
\hline 
Av. Freq. $MAD_e$-para & 0 & 0.008  &  0.007 &  0.008& 0.010 & 0.007 & 0.007&   0.007 \\
& 5 & 1  &  1 &  1& 1 & 0.061 & 0.441&   0.021  \\
& 10 & 1  &  1 &  1& 1 & 0.060 & 0.450&   0.022 \\
Av. Numb. & 5 & 5   &  5&  5& 5 & 1.03 & 1.32 & 1.01\\
 & 10 & 10 & 9.97& 10& 10& 1.03 & 1.34& 1.01 \\
\hline 
\hline
\end{tabular}
\end{center}
\end{table}

\begin{table}
\caption{ \label{Table9} Average frequencies (Av. Freq.) of potential outliers with the different outlier detectors and average number (Av. Numb.) of detected outliers when outliers are detected, for $n=100$ and $n=1000$, with $K$ point outliers. Here $\alpha=0.007$ and $20000$ independent replications are generated.} 
\begin{center}
\footnotesize\begin{tabular}{|l|c||c|c|c|c|c|c|c|}
\hline\hline
 $n=100$ &$K$ & $\big | {\cal N}(0,1) \big |$& ${\cal E}(1)$ & $\Gamma(3)$ &  $W(3,4)$ & $| t(2) |$ & $\log- {\cal N}(0,1)$  & $|{\cal C}|$\\
\hline \hline
Av. Freq. $\widetilde D_{J_n}$ & 0 & 0.007 &  0.008 &  0.008& 0.008 & 0.010 & 0.010&   0.018  \\
& 5 & 1  &  1 & 1& 1 & 1 & 1&   0.719 \\
& 10 & 1  & 1 & 1 & 1 & 1 & 1&  0.979 \\
Av. Numb. & 5 & 5.11 & 5.14&  5.19& 5.15 & 5.43  & 5.27 & 5.83\\
 & 10 & 10.63  & 10.67 & 10.75& 10.68& 11.10 & 10.98 & 11.25\\
\hline 
Av. Freq. Student-para & 0 & 0.002 &  0 &  0.001& 0.005 & - & 0& -  \\
& 5 & 0 &  0 &  0& 1 & - & 0&   -  \\
& 10 & 0 & 0 & 0& 0 & - & 0  &  -  \\
Av. Numb. & 5 &0 &  0&  0& 5 & -  & 0 & - \\
 & 10 & 0   & 0 & 0 & 0 & -  & 0 & -\\
\hline 
Av. Freq. Tukey-para & 0 & 0.022 &  0.013 &  0.014& 0.025 & 0.008 & 0.009&   0.007  \\
& 5 & 1  &  1 &  1& 1 & 1 & 1 &   0.005  \\
& 10 & 1  &  1 &  1& 1 & 1 & 1 &   0.004 \\
Av. Numb. & 5 & 5.01  &  5&  5.01& 5.01& 5 & 5& 1.01\\
 & 10 & 10   & 10 & 10& 10 & 10 & 10 & 1\\
\hline 
Av. Freq. $MAD_e$-para & 0 & 0.023  &  0.014 &  0.013& 0.025 & 0.007 & 0.009&   0.007  \\
& 5 & 1  &  1 &  1& 1 & 1& 1 &  0.006  \\
& 10 & 1  &  1 &  1& 1 & 1 & 1 & 0.006 \\
Av. Numb. & 5 & 5.01   &  5.01 &  5.01& 5.01 & 5.01 & 5.01& 1.03\\
 & 10 & 10  & 10 &  10& 10 & 10 & 10 & 1\\
\hline 
\hline
 $n=1000$ &$K$ & $\big | {\cal N}(0,1) \big |$& ${\cal E}(1)$ & $\Gamma(3)$ &  $W(3,4)$ & $| t(2) |$  & $\log- {\cal N}(0,1)$ & $|{\cal C}|$\\
\hline \hline
Av. Freq. $\widetilde D_{J_n}$ & 0 & 0.009 &  0.009 &  0.009& 0.009 & 0.014 & 0.011&   0.016  \\
& 5 & 1  &  1&  1& 1 & 1& 1&   0.246 \\
& 10 & 1  &  1 &  1& 1 & 1 & 1 &   0.709 \\
Av. Numb. & 5 & 5.23 &  5.27 &  5.32& 5.26 & 5.69 & 5.28& 7.76\\
 & 10 & 11.03  &  11.11 & 11.22& 11.13 & 11.09 & 11.35& 12.73\\
\hline 
Av. Freq. Student-para & 0 & 0.006  &  0.004 &  0.005& 0.007 & - &0 & - \\
& 5 & 1  &  1 &  1& 1 & - & 0 &  - \\
& 10 & 1  & 1 &  1& 1 & - & 0 &  - \\
Av. Numb. & 5 & 5  & 5 &  5& 5 & -  & 0 & - \\
 & 10 & 10  & 10  & 10& 10& - & 0 & - \\
\hline 
Av. Freq. Tukey-para & 0 & 0.008  &  0.008 &  0.008& 0.010 & 0.007 & 0.007&   0.007  \\
& 5 & 1  &  1 &  1& 1 & 1 & 1 &   0.007 \\
& 10 & 1  &  1 &  1& 1 & 1 & 1&   0.007 \\
Av. Numb. & 5 & 5.01 &  5.01&  5.01& 5.01 & 5.01 & 5.01& 1\\
 & 10 & 10.01  & 10.01 & 10.01& 10.01 & 10.01 & 10.01 & 1.01\\
\hline 
Av. Freq. $MAD_e$-para & 0 & 0.008  &  0.007 &  0.008& 0.010 & 0.007 & 0.007&   0.007 \\
& 5 & 1  &  1 &  1& 1 & 1& 1 &   0.007\\
& 10 & 1  &  1 &  1& 1 & 1 & 1 &   0.008 \\
Av. Numb. & 5 & 5.01   &  5.01 &  5.01& 5.01 & 5.01 & 5.01 & 1\\
 & 10 & 10.01 & 10.01& 10.01& 10& 10.01  & 10.01 & 1.01 \\
\hline 
\hline
\end{tabular}
\end{center}
\end{table}
\subsubsection*{Conclusions of simulations}
The log-ratio detector $\widehat D_{J_n}$ provide a very good trade-off for the detection of outliers for a very large choice of probability distributions compared with usual outliers detectors or their fitted versions. The second non-parametric outlier detector, LOF, is not really as accurate (especially for $K=10$), its size is not controlled and results did not depend on the choice of $\alpha$ (a theoretical study should be done for writing the threshold as a function of $\alpha$). 
\section{Application to real data}\label{applis}
We apply the theoretical results to real datasets of detailed data on individual transactions in the used car market. The purpose of the experiment was to detect as many outliers as possible. The original dataset  contains information about $n=6079$ transactions on the car  \textit{Peugeot 207 1.4 HDI 70 Trendy Berline } including  year and month which is the manufacture date, the price, and the number of kilometres driven. We then define $3$ variables: Age (the age of the car, in months), Price (in euros) and Mileage (in km). We chose these cars because they were advertised often enough to permit us to create a relatively homogeneous sample. Figure \ref{Data_Peug207} depicts the relationship between the price and some variables:  Price with Mileage, Price with Age. Such data were collected by Autobiz society, and can be used for forecasting the price of a car following its age and mileage. Hence it is crucial to construct a model for the price from a reliable data set including the smallest number of outliers. \\
\begin{figure}
\begin{center}
\includegraphics[scale=0.25]{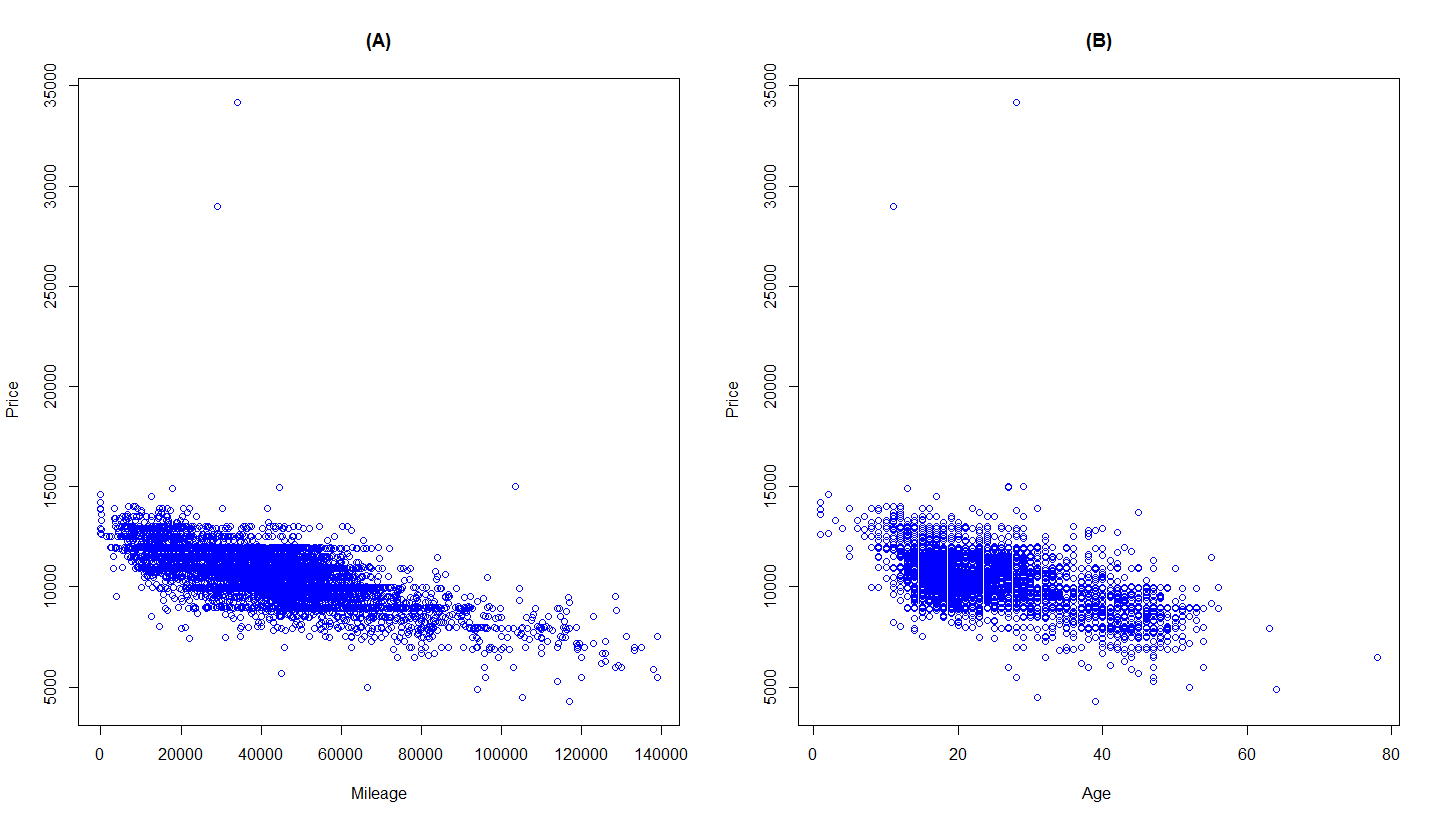}
\caption{\label{Data_Peug207}\textit{Relationship between the dependent variables and the regressors: Price with Mileage (left), Price with Age (right). }} 
\end{center}
\end{figure}
We now apply our test procedure to identify eventual outlying observations or atypical combination between variables. After preliminary studies, we chose two significant characteristics for each car of the sample. The first one is the number of kilometres per month. The second one is the residual obtained, after an application of the exponential function, from a robust quantile regression between the logarithm of the price as the dependent variable and Age and Mileage as exogenous variables (an alternative procedure for detecting outliers in robust regression has been developed in Gnanadesikan and Kettenring, 1972). The assumption of independence is plausible for both these variables' residuals. Figure \ref{Data_Peug207_Analysed} exhibits the boxplots of the  distributions of those two variables.  \\
\begin{figure}
\begin{center}
\includegraphics[scale=0.6]{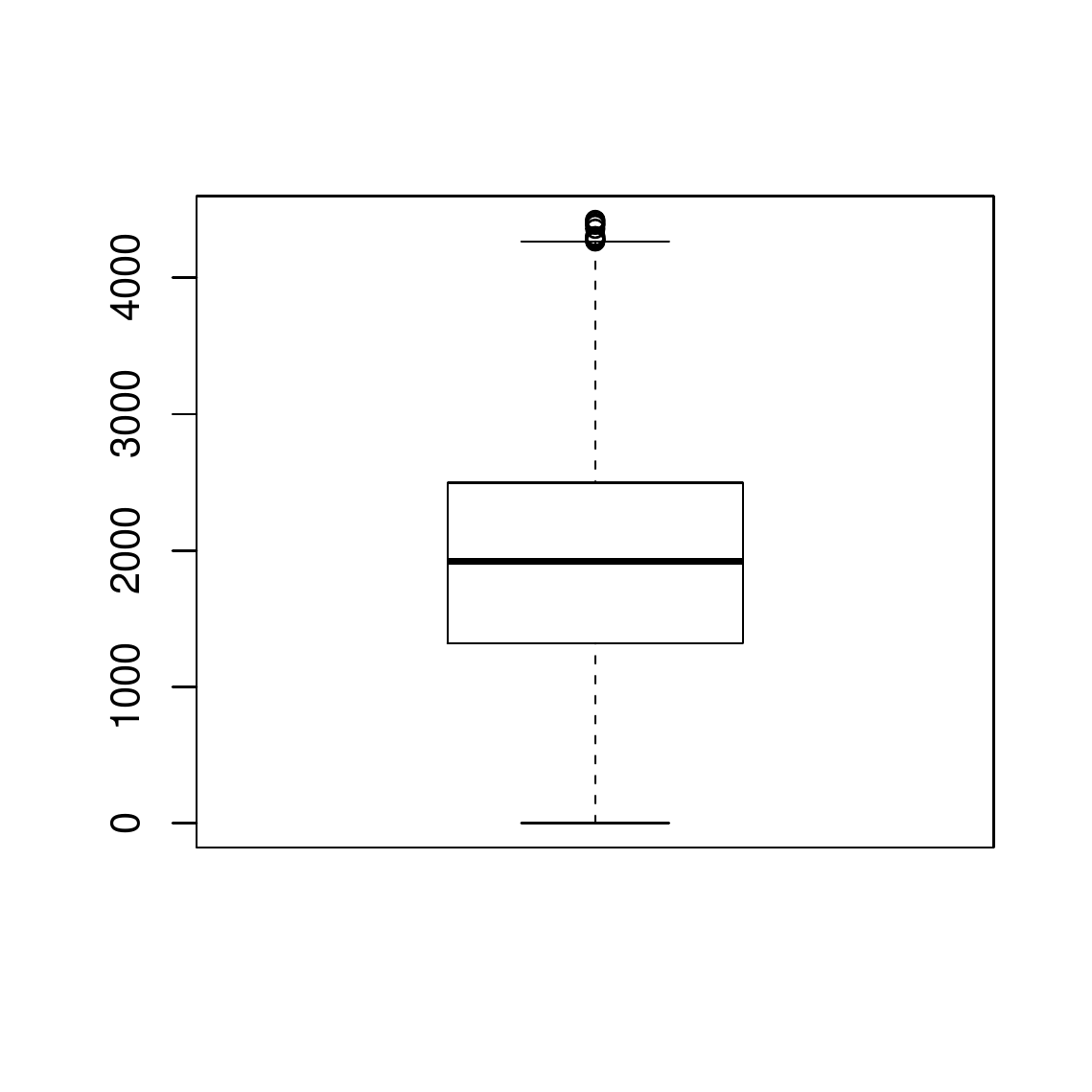} \includegraphics[scale=0.5]{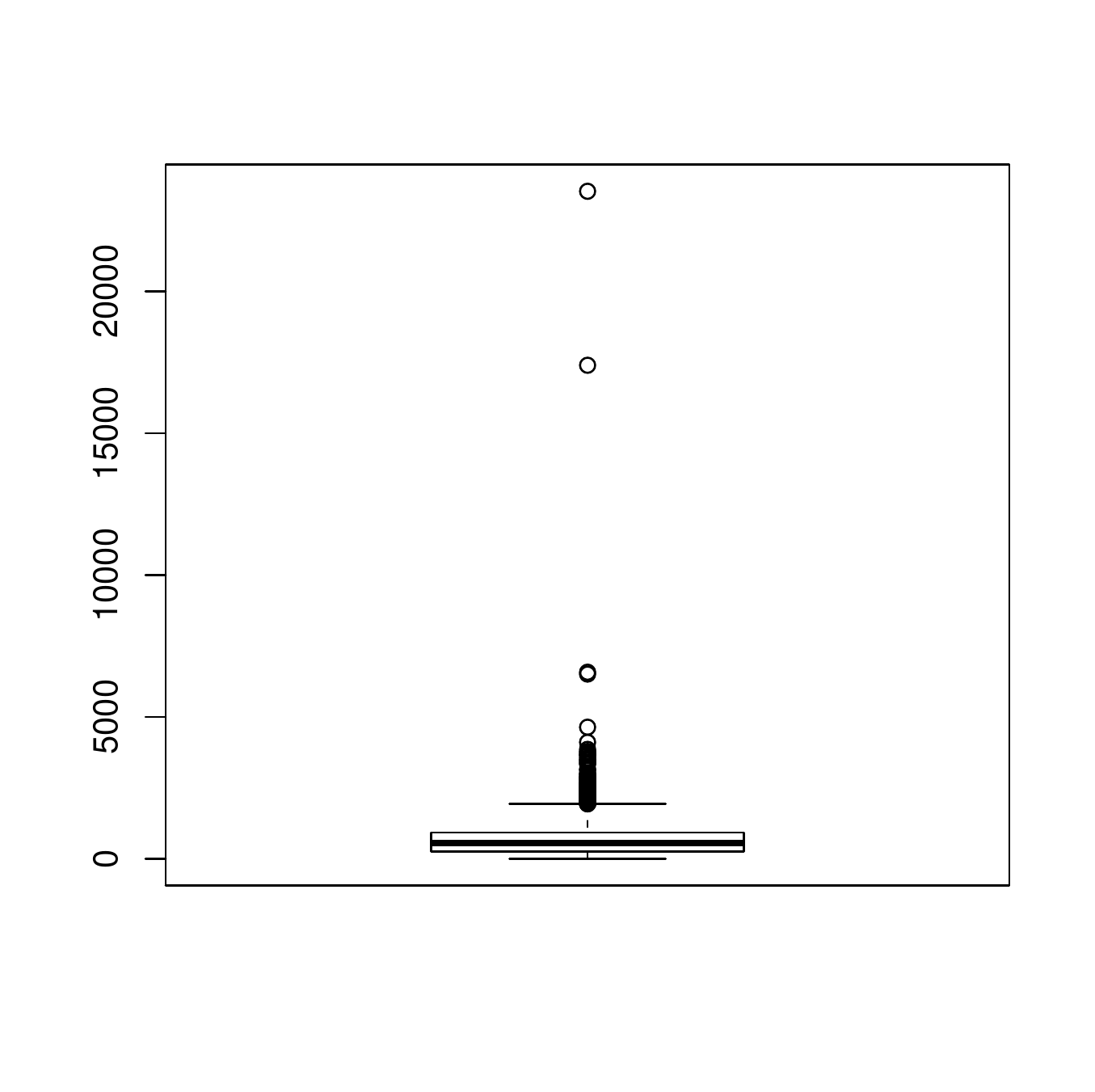}
\caption{\label{Data_Peug207_Analysed}\textit{Boxplots of kilometres per month (left) and of absolute values of   quantile regression residuals (right).}}
\end{center}
\end{figure}
The outlier test $\widehat D_{J_n}$ is carried out on those two variables with $J_n=20$ (given by the empirical choice obtained in Section  \ref{simu} with $n=6079$). As the sample size is large, we can accept to eliminate data detected as outliers while there are not really outliers and we chose $\alpha=0.05$. The results are presented in Tables \ref{ResultatTest}, \ref{Outliers_Kmm}  and \ref{Outliers_Residuals}. Note that, concerning the study of kilometres per month (km/m), we directly applied the test to this variable for detecting eventual "too" large values, but also to $\max(km/m)-(km/m)$ for detecting eventual "too" small values.
\subsubsection*{Conclusions of the applications} 
We first remark that we did not get the same outliers from the different analysis. It could be expected  because the test on residuals worked as a multivariate test and identify  atypical association between the three variables Age, Mileage and Price while the tests done on kilometres/month identifies outlying values in a bivariate case {\it i.e.} a typical association between the two variables Age and Mileage. From a practitioner's point of view it may be advisable to apply the test for the two cases together one by one to be sure to detect the largest number of outliers. A second remark concerns the "type" of the detected outliers. We can state that concerning kilometres/month, outliers are simply the largest values (the test did not identify outliers for "too" small values). But for the regression residuals, the detected outliers clearly correspond to typing errors on the prices (the prices have been replaced by the mileages!). Thus, two kinds of outliers have been detected. 
\begin{table}
\caption{\label{ResultatTest} The outlier test $\widehat D_{J_n}$ applied to $3$ samples: the number of kilometres per month ($km/m$), $\max(km/m)-km/m$ and the residuals obtained from a quantile regression of the log-prices onto the age and the mileage.}
\begin{center}
\begin{tabular}{|l|l|l|c|r|} 
\hline  Sample & $ J_n$ & $\widehat{D}_{J_n}$ &  t  & Outliers \\
\hline  km/m (Sup)  & 20 & 6.7232 & 5.96721 & $ n=6 $ \\ 
\hline  km/m (Inf)  & 20 & 5.1200 & 5.96721 & $n=0 $ \\
\hline  Res & 20 & 6.3322 & 5.96721 & $ n=2 $ \\ 
\hline
\end{tabular} 
\end{center}
\end{table}
~\\
\begin{table}
\caption{\label{Outliers_Kmm} Detailed analysis of the detected outliers obtained from the sample of kilometers per month (large values).}
\begin{center}
\begin{tabular}{|l|l|l|c|r|r|}
\hline Detected Outliers & Price & Mileage & Age & Kilometers per Month & Predicted Price \\
\hline outlier(1) &9590	&70249&	16&	4391&	9909 \\
\hline outlier(2) &11690&	61484&	14&	4392&	10286 \\ 
\hline outlier(3) &10490&	61655&	14&	4404&	10280 \\ 
\hline outlier(4) &9390	&61891&	14&	4421&	10272 \\ 
\hline outlier(5) &11500&	39826&	9	&4425&	11285 \\ 
\hline outlier(6) &11900&	65411&	15&	4361&	10111 \\ 
\hline 
\end{tabular} 
\end{center}
\end{table}
~\\
\begin{table}
\caption{\label{Outliers_Residuals} Detailed analysis of outliers detected from the residual's sample.}
\begin{center}
\begin{tabular}{|l|l|l|c|r|}
 \hline Detected Outliers & Price &	Mileage	& Age & 	Predicted Price \\
 \hline Outlier(1) & 34158 &	34158	& 28	& 10626 \\
 \hline Outlier(2) & 29000	& 29000	 & 11	 & 11600 \\
\hline 
\end{tabular} 
\end{center}
\end{table}

\newpage
\section{Proofs} \label{proofs}
\begin{proof}[Proof of Proposition \ref{prop1}]
We begin by using the classical following result (see for example Embrechts {\it et al.} 1997):
\begin{eqnarray}\label{emb}
\Big ( X_{(n-J)},X_{(n-J+1)},\cdots, X_{(n)} \Big )\stackrel{\cal D}{=} \Big ( G^{-1}\big ( \Gamma_{J+1 } / \Gamma_{n +1 } \big ) ,G^{-1}\big ( \Gamma_{J } / \Gamma_{n +1 }  \big), \cdots, G^{-1}\big ( \Gamma_{1} / \Gamma_{n +1 }  \big) \Big ),
\end{eqnarray}
where $(\Gamma_i)_{i \in \N^*}$ is a sequence of random variables such as $\Gamma_i=E_1+\cdots+E_i$ for $i \in \N^*$ and $(E_i)_{j\in \N^*}$ is a sequence of i.i.d.r.v. with distribution ${\cal E}(1)$. Consequently, we have
$$
\Big ( \tau_{n-J},\tau_{n-J+1},\cdots, \tau_{n-1} \Big )\stackrel{\cal D}{=} \Big (\frac { G^{-1}\big ( \Gamma_{J} / \Gamma_{n +1 } \big ) } { G^{-1}\big ( \Gamma_{J +1 } / \Gamma_{n +1 } \big ) },\frac { G^{-1}\big ( \Gamma_{J -1} / \Gamma_{n +1 } \big ) } { G^{-1}\big ( \Gamma_{J } / \Gamma_{n +1 } \big ) }, \cdots, \frac { G^{-1}\big ( \Gamma_{1 } / \Gamma_{n +1 } \big ) } { G^{-1}\big ( \Gamma_{2 } / \Gamma_{n +1 } \big ) }\big) \Big).
$$
But for $j \in \N^*$, $\displaystyle G^{-1}\big ( \Gamma_{j} / \Gamma_{n +1 } \big )=G^{-1}\Big (\frac 1 {\Gamma_{n+1}} \times \Gamma_{j} \Big )$. From the strong law of large numbers, $\Gamma_{n+1} \limiteasn \infty$, therefore since  $G^{-1} \in A_1$, we almost surely obtain:
\begin{eqnarray*}
G^{-1}\big ( \Gamma_{j} / \Gamma_{n +1 } \big )=f_1\big (\frac 1 {\Gamma_{n+1}} \big )\times \Big (1 - \frac {f_2(\Gamma_{j} )}{\log(\Gamma_{n+1} )}+ O\big ( \frac 1 {\log^2 (\Gamma_{n+1})} \big )\Big ).
\end{eqnarray*}
Using once again the strong law of large numbers, we have $\Gamma_{n+1} \sim n \limiten \infty$ almost surely. Hence, we can write for all  $j=1,\cdots, J$,
\begin{eqnarray}
\nonumber \frac { G^{-1}\big ( \Gamma_{j} / \Gamma_{n +1 } \big ) } { G^{-1}\big ( \Gamma_{j +1 } / \Gamma_{n +1 } \big ) } &=&
\frac {1 - \frac {f_2(\Gamma_{j} )}{\log(\Gamma_{n +1 } )}+ O\big ( \frac 1 {\log^2 (\Gamma_{n +1 } )} \big )}{1 - \frac {f_2(\Gamma_{j+1} )}{\log(\Gamma_{n +1 } )}+ O\big ( \frac 1 {\log^2 (\Gamma_{n +1 } )} \big )} \\
\label{dl1} &=&1 + \frac {f_2(\Gamma_{j+1} )-f_2(\Gamma_{j} )}{\log(\Gamma_{n +1 } )}+ O\big ( \frac 1 {\log^2 (\Gamma_{n +1 } )} \big ),
\end{eqnarray}
using a Taylor expansion.  
By considering now the family $(\tau'_j)_j$ defined by $\tau'_j=(\tau_j-1) \log (n)$ and the limit of the previous expansion, we obtain 
\begin{equation}\label{tauprime}
\big (\tau'_{n-J},\tau'_{n-J+1},\cdots, \tau'_{n-1} \big ) \limiteloin \Big (f_2(\Gamma_{J+1} )-f_2(\Gamma_{J} )\, , \, f_2(\Gamma_{J} )-f_2(\Gamma_{J-1} )\, , \cdots,\,  f_2(\Gamma_{2} )-f_2(\Gamma_{1} ) \Big ) .
\end{equation}
The function  $(x_1,\cdots,x_J) \in \R ^J \mapsto \max (x_1,\cdots,x_J) \in \R$ is a continuous function on $\R^J$ and therefore we obtain \eqref{asympto1}.
\end{proof}
~\\
\begin{proof}[Proof of Proposition \ref{prop2}]
We use the asymptotic relation \eqref{asympto1}. Since $G^{-1} \in A'_1$, for $k=1,\cdots,J$,
$$
f_2 (\Gamma_{k+1})-f_2 (\Gamma_{k})=C_2 \, \log \big (\Gamma_{k+1}/\Gamma_k \big )=C_2 \,  \log \big (\Gamma_{k+1} / \Gamma_{J+1} \big )-C_2 \, \log \big (\Gamma_k / \Gamma_{J+1} \big ), 
$$ 
But the random variable $f_2 (\Gamma_{k+1})-f_2 (\Gamma_{k})=f_2(\Gamma_k+E_{k+1})-f_2(\Gamma_k)$ is absolutely continuous with respect to Lebesgue measure since $\Gamma_k$ and $E_{k+1}$ are independent random variables.
Using once again the property \eqref{emb}, and since  for an exponential distribution ${\cal E}(1)$, $G^{-1}(x)=-\log(x)$, then 
$$
\Big (-\log\big (\Gamma_J/\Gamma_{J+1}\big ) \, ,\, -\log\big (\Gamma_{J-1}/\Gamma_{J+1}\big )\, , \cdots,\,  -\log\big (\Gamma_1/\Gamma_{J+1}\big )\Big )
\stackrel{\cal D}{=}  \Big (E'_{(1)},E'_{(2)},\cdots,E'_{(J)}\Big ).
$$
where $(E'_j)_j$ is a sequence of i.i.d.r.v. following a  ${\cal E}(1)$ distribution and $E'_{(1)}\leq  E'_{(2)} \leq \cdots \leq E'_{(J)}$ is the order statistic from $(E'_1,\cdots,E'_J)$. Consequently,
\begin{eqnarray*}
&& \hspace{-0.7cm}\Big (f_2 (\Gamma_{J+1})-f_2 (\Gamma_{J}),\, f_2(\Gamma_{J} )-f_2(\Gamma_{J-1} )\, , \cdots,\,  f_2(\Gamma_{2} )-f_2(\Gamma_{1} ) \Big ) \\
  &&\hspace{0.7cm} =C_2  \Big (-\log \big (\Gamma_{J} / \Gamma_{J+1} \big )\, , \,\log \big (\Gamma_{J} / \Gamma_{J+1} \big )-\log \big (\Gamma_{J-1} / \Gamma_{J+1} \big )\,, \cdots, \,\log \big (\Gamma_{2} / \Gamma_{J+1} \big )-\log \big (\Gamma_1 / \Gamma_{J+1} \big ) \Big ) \\
& & \hspace{0.7cm}\stackrel{\cal D}{=} C_2 \, \Big (E'_{(1)}\, , \, E'_{(2)}-E'_{(1)}\, ,\, \cdots,E'_{(J)}-E'_{(J-1)} \Big )
\end{eqnarray*}
With $y=x/C_2$ and using \eqref{asympto1}, this implies 
\begin{eqnarray*}
\P \big (\max_{j=n-J,\cdots,n-1} \{\tau_j'\} \leq x \big )  &\limiten& \P \big ( E'_{(1)} \leq y, \, E'_{(2)} \leq  y+E'_{(1)}, \cdots, E'_{(J)} \leq y+E'_{(J-1)} \big ) \\
 &\limiten& J! \, \, \P \big (  E'_{1} \leq y, \, E'_{1} \leq E'_{2} \leq y+E'_{1}, \cdots, E'_{J-1} \leq E'_{J} \leq y+E'_{J-1}   \big ). 
\end{eqnarray*}
The explicit computation of this probability is possible. Indeed:
\begin{eqnarray*}
&& \P \big (  E'_{1} \leq y, \, E'_{1} \leq E'_{2} \leq y+E'_{1}, \cdots, E'_{J-1} \leq E'_{J} \leq y+E'_{J-1}  \big ) \\
&& \hspace{2cm} = \int_{0}^y \int_{e_1}^{y+e_1}\int_{e_2}^{y+e_2} \cdots \int_{e_{J-2}}^{y+e_{J-2}} \int_{e_{J-1}}^{y+e_{J-1}}  e^{-e_1}   e^{-e_2}  e^{-e_3}   \cdots e^{-e_{J-1}}e^{-e_{J}} de_1 de_2 de_3 \cdots     de_{J-1}   de_{J} \\
&& \hspace{2cm} =\big ( 1 - e^{-y} \big )   \int_{0}^y \int_{e_1}^{y+e_1}\int_{e_2}^{y+e_2} \cdots  \int_{e_{J-2}}^{y+e_{J-2}}  e^{-e_1}  e^{-e_2}    e^{-e_3} \cdots     e^{-2e_{J-1}} de_1de_2 de_3  \cdots de_{J-1} \\
&& \hspace{2cm} = \frac 1 2 \big ( 1 - e^{-y} \big ) \big ( 1 - e^{-2y} \big )   \int_{0}^y  \int_{e_1}^{y+e_1} \int_{e_2}^{y+e_2} \cdots  \int_{e_{J-3}}^{y+e_{J-3}}   e^{-e_1} e^{-e_2}   e^{-e_3}  \cdots     e^{-3e_{J-2}} de_1 de_2 de_3  \cdots de_{J-2} \\
&& \hspace{2cm} = \qquad \qquad \vdots \qquad \qquad  \vdots \qquad \qquad  \vdots \qquad \qquad \vdots \qquad  \qquad \vdots \qquad \qquad \vdots \qquad  \qquad \vdots \\
&& \hspace{2cm} = \frac 1 {(J-2)!} \,  \big ( 1 - e^{-y} \big ) \big ( 1 - e^{-2y} \big )\times  \cdots \times  \big ( 1 - e^{-(J-2)y} \big )  \int_{0}^y \int_{e_1}^{y+e_1} e^{-e_1}   e^{-(J-1)e_2} de_1 de_2 \\
&& \hspace{2cm} =  \frac 1 {(J-1)!} \,  \big ( 1 - e^{-y} \big ) \big ( 1 - e^{-2y} \big )\times  \cdots \times  \big ( 1 - e^{-(J-1)y} \big )  \int_{0}^y e^{-Je_1} de_1 \\
&& \hspace{2cm} =  \frac 1 {J!} \,  \big ( 1 - e^{-y} \big ) \big ( 1 - e^{-2y} \big )\times  \cdots \times  \big ( 1 - e^{-Jy} \big ) .
\end{eqnarray*}
Then, we obtain \eqref{asympto2}.

\end{proof}
~\\
\begin{proof}[Proof of Proposition \ref{prop3}]
Such a result can be obtained by modifications of  Propositions \ref{prop1} and \ref{prop2}. Indeed, we begin by extending Proposition \ref{prop1} in the case where $J_n \limiten \infty$ and $J_n/\log n \limiten 0$. This is possible since $\Gamma_{n+1} /n =1 + n^{-1/2}\varepsilon_n$ with $\varepsilon_n \limiteloin {\cal N}(0,1)$ from usual Central Limit Theorem. Using the Delta-method, we also obtain $\log (\Gamma_{n+1} /n)=n^{-1/2}\varepsilon'_n$ with $\varepsilon'_n \limiteloin {\cal N}(0,1)$. Hence, for any $j=1,\cdots,J_n$, from \eqref{tauprime},
\begin{eqnarray*}
\tau'_{n-j} &\stackrel{\cal D} =&\log (n) \, \Big ( \frac { G^{-1}\big ( \Gamma_{j} / \Gamma_{n +1 } \big ) } { G^{-1}\big ( \Gamma_{j +1 } / \Gamma_{n +1 } \big )} -1 \Big ) \\
&\stackrel{\cal D} =&f_2(\Gamma_{j+1} )-f_2(\Gamma_{j} )+ O\big ( \frac 1 {\log (n)} \big )
\\
&\stackrel{\cal D} =&C_2\, \log(\Gamma_{j+1} /\Gamma_{j_n} )+ O\big ( \frac 1 {\log (n)} \big ).
\end{eqnarray*}
Denote $F_n$ the cumulative distribution function of $\big (\tau'_{n-J_n},\cdots,\tau'_{n-1} \big )$, and  $\widetilde F_n$ the one of $\big (f_2(\Gamma_{J_n+1} )-f_2(\Gamma_{J_n} ),\cdots,f_2(\Gamma_{2} )-f_2(\Gamma_{1} ) \big )=C_2\big (\log(\Gamma_{J_n+1} /\Gamma_{J_n} ),\cdots,\log(\Gamma_{2} /\Gamma_{1} ) \big )$. Then, using the second equation of the proof of Proposition \ref{prop1}, for  all $(x_1,\cdots,x_{J_n}) \in (0,\infty)^{J_n}$, 
$$
F_n(x_1,\cdots,x_{J_n})=\widetilde F_n (x_1+u^1_n,\cdots,x_{J_n}+u^{J_n}_n),
$$
with $u^i_n=O\big ( \frac 1 {\log (n)} \big )$. But it is clear that the probability measure of  $\big (f_2(\Gamma_{J_n+1} )-f_2(\Gamma_{J_n} ),\cdots,f_2(\Gamma_{2} )-f_2(\Gamma_{1} ) \big )$ is absolutely continuous with respect to the  Lebesgue measure on $R^{J_n}$. Thus, the partial derivatives of the function $\widetilde F_n$ exist. Then from a Taylor-Lagrange expansion,
$$
\widetilde F_n (x_1+u^1_n,\cdots,x_{J_n}+u^{J_n}_n) = \widetilde F_n (x_1,\cdots,x_{J_n})+ \sum_{j=1}^{J_n} u^{j}_n \times \frac {\partial}{\partial x_j}F_n (x'_1,\cdots,x'_{J_n}),
$$
where $(x'_1,\cdots,x'_{J_n}) \in (0,\infty)^{J_n}$. Hence, we obtain $\Big |\sum_{j=1}^{J_n} u^{j}_n \times \frac {\partial}{\partial x_j}F_n (x'_1,\cdots,x'_{J_n})\Big | \leq C \sum_{j=1}^{J_n} u^{j}_n \leq C' \frac{J_n}{\log n}$ for some positive real numbers $C$ and $C'$. Consequently, we have:
$$
F_n(x_1,\cdots,x_{J_n}) \simn  \widetilde F_n (x_1,\cdots,x_{J_n}).
$$
Now, we are going back to the proof of Proposition \ref{prop2} by computing $\widetilde F_n (x_1,\cdots,x_{J_n})$. This leads to compute the following integral:
$$
\int_{0}^{y_1} \int_{e_1}^{y_2+e_1} \int_{e_2}^{y_3+e_2} \cdots \int_{e_{J_n-2}}^{y_{J_n-1}+e_{J_n-2}} \int_{e_{J_n-1}}^{y_{J_n}+e_{J_n-1}}   e^{-e_1}  e^{-e_2}    e^{-e_3}  \cdots  e^{-e_{J_n-1}}    e^{-e_{J_n}} de_1de_2de_3 \cdots de_{J_n-1} de_{J_n},
$$
with $y_i=x_i/C_2$, and with the same iteration than in the proof of  Proposition \ref{prop2}, we obtain
$$
\widetilde F_n (x_1,\cdots,x_{J_n})= \prod_{j=1}^{J_n} \big ( 1 - e^{-jx_{J_n-j+1}/C_2}\big ).
$$
Then, by considering the vector $((n-j)\tau'_j)_{n-J_n\leq j \leq n-1}$ and the continuity of the function $\max$, we have for all $x\geq 0$
\begin{equation}\label{loi}
\Pr \big ( \max_{j=n-J_n,\cdots,n-1} \{(n-j)\tau'_j\} \leq x \big ) \simn  \big ( 1 - e^{-x/C_2}\big )^{J_n}.
\end{equation}
To achieve the proof, we use the Slutsky's Theorem. Indeed, we have $\overline{s}_{J_n}=\frac 1 {J_n} \, \sum_{j=1}^{J_n} j\,  \tau'_{n-j}$ with $J_n \to \infty$. As a consequence,
\begin{eqnarray}
\overline{s}_{J_n}&=& \frac {C_2} {J_n} \, \sum_{j=1}^{J_n} j \log\big (\Gamma_{j+1} /\Gamma_{j} \big )+ O\big ( \frac {\sum_{j=1}^{J_n} j} {J_n\log (n)} \big )\\
&= &\frac {C_2} {J_n} \, \sum_{j=1}^{J_n} j \log\big (\Gamma_{j+1} /\Gamma_{j} \big )+ O\big ( \frac {J_n} {\log (n)} \big ) \\
& =& \frac {1} {J_n} \, \sum_{j=1}^{J_n} \varepsilon_j+ O\big ( \frac {J_n} {\log (n)} \big ) ,
\end{eqnarray}
with $(\varepsilon_1,\cdots,\varepsilon_{J_n})$ a family of i.i.d.r.v. with  exponential distribution of parameter $1/C_2$. Using the Law of Large Numbers and the condition $J_n/\log (n) \limiten 0$, we deduce that
\begin{equation}\label{lgn}
\overline{s}_{J_n} \limiteproban C_2.
\end{equation}
Then the proof can be concluded using \eqref{loi}, \eqref{lgn} and Slutsky's Theorem. 
\end{proof}
~\\
\begin{proof}[Proof of Proposition \ref{prop4}]
We begin by considering the proof of Proposition \ref{prop1}. Hence, since $G^{-1} \in A_2$ and always with $\Gamma_{n+1} \simn n$ almost surely, we obtain for $k=1,\cdots,J$,
\begin{eqnarray*}
\log \Big ( \frac {G^{-1} \big (\Gamma_k/\Gamma_{n+1} \big ) }{G^{-1} \big (\Gamma_{k+1}/\Gamma_{n+1} \big ) } \Big )&=& \log \Big ( \frac {\Gamma_k^{-a} g_1(1/\Gamma_{n+1})\big (1+ O(1/\log(\Gamma_{n+1})) \big ) }{\Gamma_{k+1}^{-a} g_1(1/\Gamma_{n+1})\big (1+ O(1/\log(\Gamma_{n+1}) )\big )}\Big )\\
&=& -a \, \log \big ( \Gamma_k/\Gamma_{k+1} \big )+O(1/\log(n)).
\end{eqnarray*}
Then, we directly use the proof of Proposition \ref{prop2}. 
\end{proof}
~\\
\begin{proof}[Proof of Theorem \ref{theo}]
First consider the case $G^{-1} \in A'_1$. Using the proof of Proposition \ref{prop1}, we have
$$
\log \Big ( \frac { G^{-1}\big ( \Gamma_{j} / \Gamma_{n +1 } \big ) } { G^{-1}\big ( \Gamma_{j +1 } / \Gamma_{n +1 } \big ) } \Big ) =\frac {f_2(\Gamma_{j+1} )-f_2(\Gamma_{j} )}{\log(n)}+ O\big ( \frac 1 {\log^2 (n)} \big ).
$$
Consequently, using $G^{-1} \in A'_1$ and therefore the definition of $f_2$, we obtain:
$$
\log (\tau_j) =\frac {C_2}{\log(n)} \log \big (\Gamma_{j+1}/\Gamma_{j} \big )+ O\big ( \frac 1 {\log^2 (n)} \big ).
$$
To prove \eqref{asympto5}, it is sufficient to use again the proof of Proposition \ref{prop3}, to normalize the numerator and denominator with $\log n$ and therefore to consider $\log n \times \widehat L_{J_n}$, which converges in probability to $\log 2/C_2$ (indeed, the median of the sample converges to $\log2 / \lambda$ which is the median of the distribution of the ${\cal E}(\lambda)$ distribution).\\
When  $G^{-1} \in A_2$, we can use the same argument that the ones of the proof of Proposition \ref{prop3} with $C_2$ replaced by $a$ (the reminder $1/\log n$ obtained from the definition of  $A_2$ allows the achievement the proof when 
$J_n$ is negligible compared to $\log n$).    
\end{proof}

\vskip0.2cm
\noindent {\bf Acknowledgements.} The authors are extremely grateful to the referees for their impressive and scrupulous reports  with many relevant corrections, suggestions and comments which helped to improve the contents
 of the paper. We also would like to thank Mia Hubert for having kindly provided R routines.\\

{}
\end{document}